\documentclass[11pt]{article}
\usepackage{amsthm,amsfonts,amsmath, amscd}
\usepackage{color}
\usepackage[pdftex]{graphicx}
\usepackage{enumerate}
\swapnumbers
\usepackage[footnotesize, bf]{caption}
\theoremstyle{plain}

\newtheorem{thm}{THEOREM}[section]
\newtheorem{lm}[thm]{LEMMA}
\newtheorem{cl}[thm]{COROLLARY}
\newtheorem{prop}[thm]{PROPOSITION}

\theoremstyle{definition}
\newtheorem{defi}[thm]{DEFINITION}
\theoremstyle{definition}
\newtheorem{remark}[thm]{Remark}

\newcommand{\upchi}{\raise1pt\hbox{$\chi$}}
\newcommand{\R}{{\mathord{\mathbb R}}}

\newcommand{\Rd}{{\mathord{\mathbb R}^d}}

\newcommand{\N}{{\mathord{\mathbb N}}}

\newcommand{\hn}{{\mathord{\widehat{n}}}}

\newcommand{\grad}{\nabla}
\newcommand{\gradw}{\grad_W}

\newcommand{\argmin}{\operatornamewithlimits{argmin}}
\newcommand{\id}{{\rm id}}
\newcommand{\bmu}{\boldsymbol{\mu}}

\numberwithin{equation}{section}
\newcommand{\version}{\today}

\def\dd{{\rm d}}
\def\H{{\mathcal H}}

\def\P{{\mathcal P}}

\begin{document}

\markboth{\scriptsize{CL \version}}{\scriptsize{CC October 10, 2012}}
\def\mn{{\bf M}_n}
\def\hn{{\bf H}_n}
\def\hnp{{\bf H}_n^+}
\def\hmnp{{\bf H}_{mn}^+}
\def\h{{\cal H}}
\title{{\sc  Contraction of the proximal map and generalized convexity of the Moreau-Yosida regularization in the 2-Wasserstein metric}}
\author{
\vspace{5pt}  Eric A. Carlen$^{1}$  and Katy Craig$^{2}$ \\
\vspace{5pt}\small{Department of Mathematics, Hill Center,}\\[-6pt]
\small{Rutgers University,
110 Frelinghuysen Road
Piscataway NJ 08854-8019 USA}\\
 }

\date{October 10, 2012}
\maketitle

\footnotetext [1]{Work partially supported by U.S.
National Science Foundation grant  DMS 0901632.}

\footnotetext [2]{Work partially supported by a Presidential Fellowship at Rutgers University \hfill\break
\copyright\, 2012 by the authors. This paper may be reproduced, in its
entirety, for non-commercial purposes.
}

\begin{abstract}
We investigate the Moreau-Yosida regularization and the associated proximal map in the context of discrete gradient flow for the $2$-Wasserstein metric. Our main results are a 
stepwise contraction property for the proximal map and an ``above the tangent line'' inequality for the regularization. Using the latter, we prove a Talagrand inequality and an HWI inequality for the regularization, under appropriate hypotheses. In the final section, the results are applied to study the discrete gradient flow  for R\'enyi entropies. As Otto showed, the gradient flow for these entropies in the $2$-Wasserstein metric is a porous medium flow or a fast diffusion flow, depending on the exponent of the entropy. 
We show that a striking number of the remarkable features of the porous medium  and  fast diffusion flows  are present in the discrete gradient flow and do not simply emerge in the limit as the time-step goes to zero. 
\end{abstract}

\medskip
\centerline{Key Words: Wasserstein metric, gradient flow, Moreau-Yosida regularization.}

\section{Introduction}

Given a complete metric space $(X,d)$, a functional $E:X \to \R \cup \{\infty \}$, and $\tau>0$, the \emph{Moreau-Yosida regularization} of $E$ is
\[E_\tau(y) := \inf_{x \in X} \left\{\frac{1}{2 \tau} d(x,y)^2 + E(x) \right\}. \]
The corresponding \emph{proximal set} $J_\tau:X \to 2^X$ is
\[J_\tau(y):= \argmin_{x \in X}  \left\{\frac{1}{2 \tau} d(x,y)^2 + E(x) \right\} . \]
If there is a unique element in $J_\tau(y)$, we denote it by $y_\tau$ and call it the \emph{proximal point}. We call $y \mapsto y_\tau$ the \emph{proximal map}.

When $X = \H$ is a Hilbert space, a suitable context in which to develop the theory of the Moreau-Yosida regularization is the class of functionals that are proper, lower semicontinuous, and convex. For all such $E$ and $\tau > 0$, the Moreau-Yosida regularization $E_\tau$ is convex and Fr\'echet differentiable \cite{M}. Furthermore, its derivative is Lipschitz continuous, and, as $\tau \to 0$, $E_\tau \nearrow E$ pointwise \cite{Brezis2}. The Moreau-Yosida regularization provides a way to regularize $E$ that preserves convexity.

The proximal map is similarly well-behaved for functionals that are proper, lower semicontinuous, and convex. For each $y \in \H$ and $\tau > 0$, there is a unique proximal point $y_\tau$, so that the proximal map $y \mapsto y_\tau$ is well-defined on all of $\H$. As shown by Moreau \cite{M}, the proximal map is a contraction in the Hilbert space norm:
\[ ||x_\tau- y_\tau|| \leq ||x - y|| \ \forall x,y \in \H . \]

One of the main reasons for interest in the Moreau-Yosida regularization and proximal map is their relation to gradient flow. The \emph{gradient flow} of a functional $E$ is the Cauchy problem
\begin{align}
\frac{d }{dt} y(t)= - \grad E(y(t)),\quad y(0) \in \overline{D(E)}=\overline{\{z \in \H : E(z) < \infty \}}, \label{Hilbertgradflow}
\end{align}
which is well-defined as long as $\grad E$ exists along the flow $y(t)$.\footnote{Alternatively, one may define the gradient flow in terms of the subdifferential \cite{Brezis}.} The Moreau-Yosida regularization plays a key role in the proof of existence for solutions to the gradient flow \cite{Brezis}. First, one uses the additional regularity of $E_\tau$ to find solutions to the related gradient flow problem
\[ \frac{d }{dt} y_\tau (t)= - \grad E_\tau(y_\tau(t)),\quad y_\tau(0)  \in \overline{D(E)}.\]
Then, as $\tau \to 0$, the curves $y_\tau(t)$ converge to a curve $y(t)$ that solves  (\ref{Hilbertgradflow}) in an appropriate sense.

The proximal map expresses the discrete dynamics of gradient flow. Specifically, one may use the proximal map to define the \emph{discrete gradient flow} sequence 
\[y_n = (y_{n-1})_\tau, \quad y_0 \in \overline{D(E)},\]
as in \cite{Ma1, Ma2}. Whenever the proximal map $y \mapsto y_\tau$ is well-defined, we may identify the proximal set $J_\tau(y)$ with its unique element $y_\tau$ and write $J_\tau^n$ to indicate $n$ repeated applications of the proximal map.
The exponential formula quantifies the sense in which the discrete gradient flow is a discretized version of gradient flow \cite{Brezis2}. If $y(t)$ is a gradient flow with initial conditions $y(0)$, then
\begin{align}
 y(t) = \lim_{n \to \infty} (J_{t/n})^n(y(0)). \label{CLformHil}
\end{align}

More recently, the Moreau-Yosida regularization and proximal map have been applied outside of the Hilbert space context to gradient flow in the 2-Wasserstein metric. Briefly, we recall some facts about this metric, mainly to establish our notation --- see \cite{AGS} and \cite{V} for more background.  We present these facts both in the most general setting, without  restrictions on the type of probability measures we consider, and in a simpler setting, focusing our attention on probability measures with finite second moment that are absolutely continuous with respect to Lebesgue measure. While our results hold in the most general setting, many interesting applications concern only the simpler setting, in which the exposition and notation is more straightforward.

Let $\P(\Rd)$ denote the set of Borel probability measures on $\Rd$. Given $\mu, \nu \in \P(\Rd)$, a Borel map $T: \Rd \to \Rd$ \emph{transports $\mu$ onto $\nu$} if $\nu(B) = \mu(T^{-1}(B))$ for all Borel sets $B \subseteq \Rd$. We call $\nu$ the \emph{push-forward of $\mu$ under $T$} and write $\nu = T \# \mu$.

Now consider a measure $\bmu \in \P(\Rd \times \Rd)$. (We will distinguish probability measures on $\Rd \times \Rd$, from probability measures on $\Rd$ by writing them in bold font.) Let $\pi_1$ be the projection onto the first component of $\Rd \times \Rd$, and let $\pi_2$ be the projection onto the second component.
The first and second \emph{marginals} of $\bmu$ are $\pi_1 \# \bmu \in \P(\Rd)$ and $\pi_2 \# \bmu \in \P(\Rd)$.

Given $\mu, \nu \in P(\Rd)$, the set of \emph{transport plans} from $\mu$ to $\nu$ is
\[ \Gamma(\mu, \nu):= \{ \bmu \in \P(\Rd \times \Rd) : \pi_1 \# \bmu = \mu \ , \pi_2 \# \bmu = \nu \} \ . \]
The \emph{2-Wasserstein distance} between $\mu$ and $\nu$ is 
\begin{align}
 W_2(\mu,\nu) := \left( \inf \left\{ \int_{\Rd \times \Rd} |x-y|^{2}d \bmu(x,y):  \bmu \in \Gamma(\mu, \nu) \right\} \right)^{1/2}. \label{W2}
 \end{align}
When $W_2(\mu,\nu)< \infty$, this infimum is attained, and we refer to the plans that attain the infimum as \emph{optimal transport plans}. We denote the set of optimal transport plans by $\Gamma_0(\mu, \nu)$.

The 2-Wasserstein distance satisfies the triangle inequality and is non-negative, non-degenerate, and symmetric. However, $\P(\Rd)$ endowed with the 2-Wasserstein distance is not a metric space, since there exist measures that are infinite distances apart. Let $\P_{\mu_0}(\Rd)$ be the subset of  $\P(\Rd)$ consisting of measures that are a finite distance from some fixed Borel probability measure $\mu_0$, so that, by the triangle inequality, $(\P_{\mu_0}(\Rd), W_2)$ is a metric space. As indicated by the notation, one may take $\mu_0$ to be the initial conditions of a gradient flow. Note that when $\mu_0 = \delta_0$, the Dirac mass at the origin, $\P_{\delta_0}(\Rd)$ is the subset of $\P(\Rd)$ with finite second moment.

We now define the 2-Wasserstein distance in a simpler setting. Let $P_2(\Rd)$ denote the set of probability measures with finite second moment and $\P_2^a(\Rd)$ denote the set of probability measures with finite second moment that are absolutely continuous with respect to Lebesgue measure. If $\mu \in P_2^a(\Rd)$ and $\nu \in \P_2(\Rd)$, the 2-Wasserstein distance between $\mu$ and $\nu$ reduces to the form
\begin{align}
W_2(\mu,\nu) := \left( \inf \left\{ \int |x-T(x)|^{2}d \mu(x):  T \# \mu = \nu \right\} \right)^{1/2}.  \label{W2reg}
\end{align}
The Brenier-McCann theorem guarantees that the infimum in (\ref{W2reg}) is attained by $T = \grad \varphi$, where $\varphi:\Rd \to \R$ is convex and $\grad \varphi$ is unique $\mu$-almost everywhere \cite{Mc}. In particular,
\[W_2^2(\mu, \nu) = \int |x-\grad \varphi(x) |^{2} d \mu(x) \ , \]
and we call $\grad \varphi$ the \emph{optimal transport map from $\mu$ to $\nu$}. To emphasize its dependence on $\mu$ and $\nu$, we denote the optimal transport map from $\mu$ to $\nu$ by $\mathbf{t}_\mu^\nu$. 

Given $\mu_1, \mu_2 \in \P(\Rd)$ with $W_2^2(\mu_1, \mu_2) < \infty$ and $\bmu \in \Gamma_0(\mu^1, \mu^2)$, a \emph{geodesic} connecting $\mu_1$ and $\mu_2 \in \P(\Rd)$ is a curve of the form
\[ \mu_\alpha^{1 \to 2}:[0,1] \to \P(\Rd), \quad \mu_\alpha^{1 \to 2} = \left( (1-\alpha)\pi_1 + \alpha \pi_2 \right) \# \bmu \ . \]
As shown in \cite[Theorem 7.2.2]{AGS}, this definition agrees with the metric space definition of a geodesic, i.e. a curve $\mu_\alpha:[0,1] \to \P(\Rd)$ with $W_2(\mu_0, \mu_1)~<~\infty$ such that $W_2(\mu_\alpha, \mu_\beta) = |\alpha - \beta|W_2(\mu_0, \mu_1)$.
If $\mu_1 \in \P_2^a(\Rd)$, $\mu_2 \in \P_2(\Rd)$, then the geodesic connecting $\mu_1$ and $\mu_2$ is unique and of the form
\[ \mu_\alpha^{1 \to 2}:[0,1] \to \P_2(\Rd), \quad \mu_\alpha^{1 \to 2} = \left( (1-\alpha)\id + \alpha \mathbf{t}_{\mu_1}^{\mu_2} \right) \# \mu_1, \]
where $\id(x) = x$ is the identity transformation.

A functional $E: \P_{\mu_0}(\Rd) \to \R \cup \{ \infty \}$ is \emph{$\lambda$-convex} in the 2-Wasserstein metric if, for all $\mu_1, \mu_2 \in \P_{\mu_0}(\Rd)$, there exists a geodesic connecting $\mu_1$ and $\mu_2$ along which $E$ is $\lambda$-convex:
\begin{align}
E(\mu_\alpha^{1 \to 2}) \leq (1-\alpha)E(\mu_1) + \alpha E(\mu_2) -  \alpha (1- \alpha)\frac{\lambda}{2}W_2^2(\mu_1,\mu_2). \label{convexdef}
\end{align}
If a functional is 0-convex, we simply call it \emph{convex}.\footnote{It is also common to refer to convex functionals in the 2-Wasserstein metric as \emph{displacement convex} \cite{Mc1}.} If a functional is 0-convex and strict inequality holds in (\ref{convexdef}) for all $\alpha \in (0,1)$, we call it \emph{strictly convex}.

 Given a functional $E: \P_{\mu_0}(\Rd) \to \R \cup \{ \infty \}$ and $\tau > 0$, its Moreau-Yosida regularization is
\begin{align}
E_\tau(\mu) := \inf_{\nu \in \P_{\mu_0}(\Rd)} \left\{\frac{1}{2 \tau} W_2^2(\mu,\nu) + E(\nu) \right\} \label{MYdef}
\end{align}
and the corresponding \emph{proximal set} $J_\tau:\P_{\mu_0}(\Rd) \to 2^{\P_{\mu_0}(\Rd)}$ is
\begin{align}
J_\tau(\mu):= \argmin_{\nu \in \P_{\mu_0}(\Rd)}  \left\{\frac{1}{2 \tau} W_2^2(\mu,\nu) + E(\nu) \right\} \ . \label{W2proxdef}
\end{align}
As before, if there is a unique element in $J_\tau(\mu)$, we denote it by $\mu_\tau$ and call it the \emph{proximal point}. Similarly, we call $\mu \mapsto \mu_\tau$ the \emph{proximal map}. The properties of the Moreau-Yosida regularization and proximal map in the 2-Wasserstein metric will be the main focus of this paper.

As in the Hilbertian case, one of the main reasons for interest in the Moreau-Yosida regularization and the proximal map in the 2-Wasserstein metric is their relation to gradient flow.  When $E$ and $\mu$ are sufficiently smooth, the \emph{2-Wasserstein gradient of $E$ at $\mu \in D(E)$} is
\begin{align}
\gradw E(\mu) = - \grad \cdot \left(\mu \grad \frac{\delta E}{\delta \rho}(\mu) \right), \label{gradw}
\end{align}
where $\frac{\delta E}{\delta \rho}$ is the functional derivative of $E$ \cite{O} \cite[Chapters 8 and 10]{AGS}.\footnote{Some authors -- e.g. \cite{AGS} -- identify the tangent vector $\gradw E(\mu)$ with the  gradient vector field $-\nabla \frac{\delta E}{\delta \rho}(\mu)$ on $\R^d$. One gets Otto's representative from this by multiplying by $\mu$ and taking the divergence. 
The choice of representatives is merely notational.}
The \emph{gradient flow} of $E$ is the Cauchy problem
\[ \frac{d }{dt}\mu(t) = - \gradw E(\mu(t)),\quad \mu(0) \in \overline{D(E)} = \overline{\{\mu \in P_{\mu_0}(\Rd) : E(\mu) < \infty \}}, \] 
which is well-defined, as long as $\gradw E(\mu(t))$ exists along the flow $\mu(t)$\footnote{Alternatively, one may define gradient flow in terms of the subdifferential \cite[Definition 11.1.1]{AGS}.}. We will sometimes refer to this as the \emph{continuous gradient flow} in order to distinguish it from the \emph{discrete gradient flow} we define below.

Otto observed that $- \grad \cdot \left(\mu \grad \frac{\delta E}{\delta \rho}(\mu) \right)$ may be viewed as the gradient vector field on the ``Riemannian manifold of probability densities on $\R^d$'' associated to the functional $E$, 
where the Riemannian metric is the infinitesimal form of the $2$-Wasserstein metric \cite{O2,O}. (It is one of his insights that the $2$-Wasserstein metric is induced by a Riemannian metric.)  In this metric, the length
of the gradient of $E$ at $\mu$ is given by
\begin{align}
|\gradw E(\mu)| = \left( \int  \left|\grad \frac{\delta E}{\delta \rho}(\mu) \right|^2 d\mu \right)^{1/2} \label{normw}\ .
\end{align}

As in the Hilbertian case,  the proximal map expresses the dynamics for discrete gradient flow. When the proximal map $\mu \mapsto \mu_\tau$ is well-defined (which occurs under much weaker assumptions on $E$ and $\mu$ than are needed to define the gradient, as we describe before equation (\ref{proxwelldefined}) below) we may define the \emph{discrete gradient flow} sequence
\begin{align}
\mu_n = (\mu_{n-1})_\tau, \quad \mu_0 \in \overline{D(E)} \ . \label{discretegradflow}
\end{align}
As before, we identify the proximal set $J_\tau(\mu)$ with its unique element $\mu_\tau$ and write $J_\tau^n$ to indicate $n$ repeated applications of the proximal map.

One of the advantages of discrete gradient flow is that it is not necessary to make precise the sense in which (\ref{gradw}) defines a gradient vector field. This fact was emphasized by De Giorgi in his theory of the metric derivative  \cite{DG} and extensively developed by Ambrosio, Gigli, and Savar\'e  \cite[Chapter 8]{AGS}.   We follow De Giorgi's lead, and all of the estimates we use involve only the length of the gradient $|\gradw E(\mu)|$. In the case that $E$ and $\mu$ lack sufficient smoothness for (\ref{normw}) to be well-defined, we will interpret the symbol $|\gradw E(\mu)|$ as the {\em metric slope}
\begin{align} \label{metricslopedef}
\limsup_{\nu \to \mu} \frac{ (E(\mu) - E(\nu))^+}{W_2(\mu,\nu)} \ .
\end{align}
We use the heuristic notation $|\gradw E(\mu)|$ since, as demonstrated by Otto \cite{O2, O}, it is often enlightening to think of $|\gradw E(\mu)|$ as coming from a Riemannian metric on $\P(\Rd)$.

In their recent book \cite{AGS}, Ambrosio, Gigli, and Savar\'e conduct a detailed study of gradient flow and discrete gradient flow in the 2-Wasserstein metric for large classes of functionals, developing the analogy with the Hilbert space theory. It would be too much to hope for a perfect analogy. For example, in the Hilbert space context, if a functional $E$ is proper, lower semicontinuous, and convex, then its Moreau-Yosida regularization $E_\tau$ is also convex. However, in the 2-Wasserstein metric, it is well-known that even when $E$ satisfies analogous assumptions, $E_\tau$ is not always convex.\footnote{For the reader's convenience, we include an example in Section \ref{proofsoftheorems}.}  The key technical difference between the two metrics is that while 
\begin{align} x \mapsto \frac{1}{2} ||x- y ||^2 \label{Hilbertconvex} \end{align}
is 1-convex along geodesics,
\begin{align} \mu \mapsto \frac{1}{2}W_2^2(\mu, \nu) \label{analogueHilbertconvex}\end{align}
is not $\lambda$-convex along geodesics, for any $\lambda \in \R$, if the dimension of the underlying space is greater than or equal to 2 \cite[Example 9.1.5]{AGS}.  Since much of De Giorgi's ``minimizing steps''  approach to gradient flow relies on the 1-convexity of (\ref{Hilbertconvex}), this lack of convexity  in the $2$-Wasserstein case complicates the implementation of De Giorgi's scheme. 

Ambrosio, Gigli, and Savar\'e circumvent this difficulty with their observation that, though $\mu \mapsto \frac{1}{2}W_2^2(\mu, \nu)$ is not 1-convex along all geodesics, it is 1-convex along a different class of curves. They define the set of \emph{generalized geodesics} to be the union of these classes of curves over all $\nu \in \P(\Rd)$ (see Section \ref{gengeosec}). By considering functionals that are \emph{convex along generalized geodesics}---a stronger condition than merely being convex along geodesics (see Section \ref{funsec})---they deduce a priori estimates that provide detailed control over the gradient flow and discrete gradient flow.

The key results that we will use concern functionals $E: \P_{\mu_0}(\Rd) \to \R \cup \{\infty\}$ that are proper, coercive, lower semicontinuous, and $\lambda$-convex along generalized geodesics (see Section \ref{funsec}).\footnote{Note that Ambrosio, Gigli, and Savar\'e often state their results in the context when $\mu_0 = \delta_0$, the Dirac mass at the origin, so $\P_{\mu_0}(\Rd) = \P_2(\Rd)$. We quote their results in broader generality, since the proofs are easily adapted to this case.} With these assumptions, Ambrosio, Gigli, and Savar\'e show that if $\tau >0$ is small enough so that $\lambda \tau >-1$, then for all $\mu \in \overline{D(E)}$ the proximal map
\begin{align}
\mu \mapsto \mu_\tau \label{proxwelldefined}
\end{align}
and the \emph{discrete gradient flow} sequence
\[\mu_n = (\mu_{n-1})_\tau, \quad \mu_0 \in \overline{D(E)},\]
are well-defined. They go on to prove the 2-Wasserstein analogue of the exponential formula (\ref{CLformHil}) relating the discrete gradient flow to the continuous gradient flow \cite[Theorem 4.0.4]{AGS}. If $\mu(t)$ is the solution to the continuous gradient flow of $E$ with initial conditions $\mu(0) \in \overline{D(E)}$, then
\begin{align}
\mu(t) = \lim_{n \to \infty} (J_{t/n})^n(\mu(0)) \ . \label{Wexpform}
\end{align}

Using the assumption of convexity along generalized geodesics,  Ambrosio, Gigli, and Savar\'e comprehensively develop the theory of continuous gradient flow. While this assumption is stronger than (standard) convexity along geodesics, it is not restrictive: all important examples of functionals that are convex along geodesics are also convex along generalized geodesics \cite[Section 9.3]{AGS}.

In this paper, we take a closer look at the Moreau-Yosida regularization and the proximal map in the 2-Wasserstein metric for functionals that are convex along generalized geodesics. We show that, while the Moreau-Yosida regularization does not preserve $E$'s convexity along all geodesics (as in the Hilbertian case), if $E$ attains its minimum at $\bar{\mu}$, the Moreau-Yosida regularization does satisfy an ``above the tangent line'' inequality at $\bar{\mu}$. This type of inequality is a necessary condition for  convexity---in particular, a function from $\R$ to $\R$ is convex if and only if it lies above its tangent line at every point.

\begin{thm}[Generalized convexity of $E_\tau$] \label{yoco}
Given $E: \P_{\mu_0}(\Rd) \to \R \cup \{\infty\}$ proper, coercive, lower semicontinuous, and $\lambda$-convex along generalized geodesics with $\lambda \geq 0$, assume that $E$ attains its minimum at $\bar{\mu}$. For $\tau >0$, define $\lambda_\tau := \frac{\lambda}{1 + \lambda \tau}$. Then for all $\mu \in \overline{D(E)}$, there exists a geodesic $\mu_\alpha^{\bar{\mu} \to \mu}$ from $\bar{\mu}$ to $\mu$ such that
\begin{align} E_\tau (\mu_\alpha^{\bar{\mu} \to \mu}) \leq (1-\alpha) E_\tau (\bar{\mu}) + \alpha E_\tau (\mu) - \alpha (1-\alpha) \frac{\lambda_\tau}{2} W_2^2(\bar{\mu},\mu). \label{lambda2convex}
\end{align} 
\end{thm}
\noindent In Section \ref{sharpsection}, we show that (\ref{lambda2convex}) is sharp by presenting an example in which $E$ is $\lambda$-convex and $E_\tau$ is no more than $\lambda_\tau$-convex

As a consequence of Theorem \ref{yoco}, we show $E_\tau$ satisfies a Talagrand inequality and an HWI inequality.

\begin{thm}[Talagrand and HWI Inequalities]\label{talthm}
Under the assumptions of the Theorem \ref{yoco}, for all $\mu \in \overline{D(E)}$, we have the Talagrand inequality
\begin{align}
E_\tau(\mu) - E_\tau(\bar{\mu}) \geq \frac{\lambda_\tau}{2}W_2^2(\mu, \bar{\mu}) \label{TalagrandInequality}
\end{align}
and the HWI inequality
\begin{align} \label{HWIInequality}
E_\tau(\mu) - E_\tau(\bar{\mu}) &\leq |\grad_W E_\tau(\mu)|W_2(\mu, \bar{\mu}) - \frac{\lambda_\tau}{2} W_2^2(\mu, \bar{\mu}) \ .
\end{align}
\end{thm}
\noindent These inequalities capture $E_\tau$'s behavior at $\bar{\mu}$ from both ends of the ``above the tangent line'' inequality.

We also develop the analogy between Hilbertian metrics and the 2-Wasserstein metric by proving a contraction inequality for the proximal map. In a Hilbert space, if $E$ is proper, lower semicontinuous, and convex, Moreau \cite{M} showed that the proximal map satisfies
\begin{align}
 ||x_\tau- y_\tau|| \leq ||x - y|| \ \forall x,y \in \H. \label{Hilbertcontraction}
\end{align}
This turns out to be a rather miraculous property of the Hilbertian norm that fails even in simple Banach spaces. For example, consider the $\ell^\infty$ norm on $\R^2$. Fix two points $a = (0,0)$ and $b=(1,1)$, and let $K$ be the closed half-space lying beneath the line $3 x_2 = x_1-4$. Let $E$ be the indicator function for $K$,
\[ E(x) := \left\{ \begin{array}{ll}
        0 & \mbox{if $x = (x_1,x_2) \in K$ }\\
       \infty & \mbox{otherwise}. \end{array} \right. \]
Then
\[J_\tau(y):= \argmin_{x \in \R^2}  \left\{\frac{1}{2 \tau} ||x - y||_\infty^2 + E(x) \right\} =  \argmin_{x \in K}  \left\{\frac{1}{2 \tau} ||x - y||_\infty^2 \right\}  \ . \]
Therefore, $J_\tau(a) = (1,-1)$ and $J_\tau(b) = (5/2, -1/2)$ for all $\tau >0$. This is not a contraction since $||a - b||_\infty = 1 < 3/2 = ||J_\tau(a) - J_\tau(b)||_\infty$.

\begin{figure}[htb]
\centering
\includegraphics[width=0.4\textwidth]{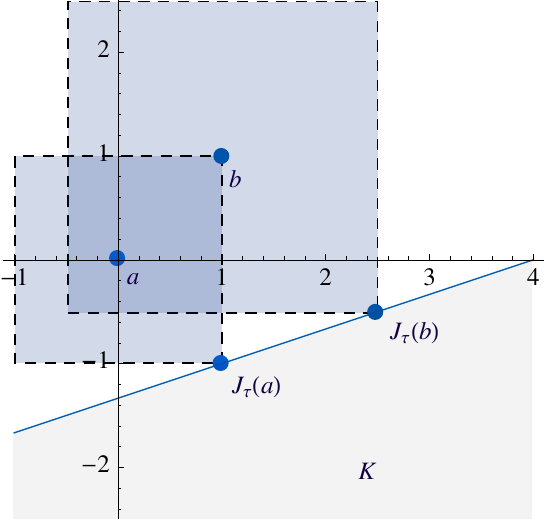}
\caption{In the Banach space $\R^2$, endowed with the $\ell^\infty$ norm, the proximal map is not a contraction.}
\end{figure}

The situation for general metric spaces is even more involved than the situation for metrics induced by norms, and one does not expect a contraction to hold. Nevertheless, if $E$ is appropriately convex, the continuous time gradient flow defined by (\ref{Wexpform}) is contractive \cite[Theorem 4.0.4]{AGS} \cite{O}. This gives hope that some contraction property of the proximal map is present at the discrete level and does not merely emerge in the limit.

Our next result shows that this is the case. In particular, we achieve contraction of the proximal map by making a small modification to the squared distance: given $\tau > 0$, we consider the functional $\Lambda_\tau: \P(\Rd) \times \P(\Rd) \to \R \cup \{ \infty \}$ defined by
\begin{align}
\Lambda_\tau(\mu,\nu) &:= W_2^2(\mu, \nu) + \frac{\tau^2}{2} |\gradw E(\mu)|^2 + \frac{\tau^2}{2} |\gradw E(\nu)|^2. \label{defofa}\ 
\end{align}
As before, we interpret $|\gradw E(\mu)|$ as the metric slope (\ref{metricslopedef}) when $E$ and $\mu$ lack sufficient smoothness for the norm of the 2-Wasserstein gradient (\ref{normw}) to be well-defined.

Though we state the following theorem in the context of the 2-Wasserstein metric, it continues to hold in a more abstract setting: given a functional $E$ on a complete metric space $(X,d)$, if $E$ is proper, coercive, lower semicontinuous, and satisfies \cite[Assumption 4.0.1]{AGS} for some $\lambda \in \R$, then the result remains true by replacing $W_2$ with $d$.

\begin{thm}[Contraction of proximal map] \label{cothm} Given $E: \P_{\mu_0}(\Rd) \to \R \cup \{\infty\}$ proper, coercive, lower semicontinuous, and $\lambda$-convex along generalized geodesics, fix $\tau >0$ small enough so that $\lambda \tau >-1$.
Consider $\mu, \nu \in \overline{D(E)}$ and let $\Lambda_\tau: \P(\Rd) \times \P(\Rd) \to \R \cup \{\infty\}$ be given by (\ref{defofa}). 
Then, if $\lambda \geq 0$, the proximal map is contracting in $\Lambda_\tau$,
\begin{align}
\Lambda_\tau(\mu_\tau, \nu_\tau) \leq \Lambda_\tau(\mu,\nu). \label{controfa}
\end{align}
More generally, for $\lambda \in \R$,

\begin{align}
\Lambda_\tau(\mu_\tau,\nu_\tau) - \Lambda_\tau(\mu,\nu) &\leq -\frac{1}{2}(\tau|\gradw E(\nu)| -W_2(\nu,\nu_\tau))^2 -\frac{1}{2}(\tau|\gradw E(\mu)| -W_2(\mu,\mu_\tau))^2 \nonumber\\
&\quad - \frac{\lambda \tau}{2} \left[ 2 W_2^2(\mu_\tau, \nu_\tau) + W_2^2(\mu, \nu_\tau) + W_2^2(\nu, \mu_\tau) + W_2^2(\nu, \nu_\tau) + W_2^2(\mu,\mu_\tau) \right]\ . \label{contr}
\end{align}
\end{thm}
In Section \ref{sharpsection}, we show that the inequality (\ref{contr}) is sharp. Then, in Section~\ref{Barenblattsection}, we apply (\ref{controfa}) together with scaling properties of the $W_2$ metric to derive sharp polynomial rates of convergence to Barenblatt profiles for certain fast diffusion and porous medium equations. Otto originally deduced these results in \cite{O} by considering a modified gradient flow problem for $\lambda$-convex functionals with $\lambda>0$. The contraction inequality (\ref{controfa}) provides a simple route to such results. The fast diffusion and porous media equations also provide examples of strictly convex functionals for which the proximal map is strictly contracting in $\Lambda_\tau$ but not in $W_2$.

\begin{remark} While Ambrosio, Gigli, and Savar\'e do not explicitly consider monotonicity results for modifications of the squared distance along the discrete gradient flow, such a result (for a different modification) can be found by reading between the lines in \cite[Lemma 4.2.4]{AGS}. Consider the alternative modification to the squared distance function defined by
\begin{equation}\label{AGSmod}
\widetilde \Lambda_\tau(\mu,\nu) :=  W_2^2(\mu, \nu) + \tau E(\mu) + \tau E(\nu) \ .
\end{equation}
If one takes the final inequality on \cite[page 92]{AGS} for $\lambda = 0$ and  $n=1$, rearranges terms, and symmetrizes in $\mu$ and $\nu$, one obtains  (\ref{controfa}) with
$\widetilde \Lambda_\tau$ in place of $ \Lambda_\tau$. A key difference between $\widetilde \Lambda_\tau$ and our functional $ \Lambda_\tau$ is that, for measures $\mu$ and $\nu$ with $|\gradw E(\mu)|$ and $|\gradw E(\mu)| < \infty$, 
$ \Lambda_\tau$ involves only an ${\mathcal O}(\tau^2)$ correction to $W_2^2(\mu, \nu)$, while  $\widetilde \Lambda_\tau$ involves  an ${\mathcal O}(\tau)$ correction to $W_2^2(\mu, \nu)$.
\end{remark}

\begin{remark}
While one might first suppose that $ \Lambda_\tau$ could only be used to study discrete gradient flows with initial data $\mu,\nu$ satisfying 
$|\gradw E(\mu)|,|\gradw E(\mu)| < \infty$, when $E$ is strictly convex, the discrete gradient flow produces this regularity in one step (see Lemma \ref{ELlem}). We shall see an example of this in Section~\ref{Barenblattsection} when we apply Theorem~\ref{cothm}
to the discrete gradient flow for the R\'enyi entropies.
\end{remark}

For $\lambda>0$, one can extract from (\ref{contr}) a useful inequality that implies, among other things,  an optimal exponential  rate of decrease of $\Lambda_\tau(\mu,\bar{\mu})$ when $E$ has a minimizer $\bar{\mu}$ (necessarily unique due to the strict convexity).

\begin{cl}[The case $\lambda > 0$] \label{lambda>0case}
Consider $\lambda > 0$ and $\tau>0$ sufficiently small so that $\tau\lambda \leq 1$. Then for all $E$ satisfying the hypotheses of Theorem 1.3 and $\mu, \nu \in \overline{D(E)}$,
\begin{equation}\label{poslam1}
(1+\tau\lambda)\Lambda_\tau(\mu_\tau,\nu_\tau) \leq 
(1-\tau\lambda)\Lambda_\tau(\mu,\nu) + 3\lambda\tau \Lambda^{1/2}_\tau(\mu,\nu)[W_2(\mu,\mu_\tau) + W_2(\nu,\nu_\tau)]\ .
\end{equation}
\end{cl}

We give the proof of this corollary in Section \ref{proofsoftheorems}. However, to explain its consequences, we state and prove a simple discrete Gronwall type inequality. 
It is a discrete version of the continuous time inequality \cite[Lemma 4.1.8]{AGS}. (See \cite{B} and \cite{E} for related discrete Gronwall inequalities.)

\begin{lm}[A discrete Gronwall type inequality]\label{yoco2} Let  $\lambda, \tau>0$,
and let $\{a_n\}$ and $\{b_n\}$ be two sequences of non-negative numbers such that
for all $n \geq 0$, 
\begin{equation}\label{poslam3}
(1+\tau\lambda)a_n\leq 
(1-\tau\lambda)a_{n-1} + \tau a_{n-1}^{1/2}b_n\ .
\end{equation}
Then,
$$a_n^{1/2}  \leq   (1+\lambda\tau)^{-n} a_0^{1/2}  +
\sqrt{\frac{\tau}{ 2\lambda }}(1+\lambda\tau)\left(\sum_{k=1}^n b_k^2\right)^{1/2}  \ .$$
\end{lm}

Consider the discrete gradient flow of $E$ starting from $\mu \in D(E)$ with $\tau > 0$ and $\tau\lambda \leq 1$. Let $\mu_0 := \mu$ and inductively define  $\{\mu_n\}$
by repeated application of the proximal map. Define $\{\nu_n\}$ in the same way, 
starting from $\nu \in D(E)$.
Now, apply  Lemma~\ref{yoco2} and Corollary~\ref{lambda>0case} to these  discrete  gradient flows of $E$, taking 
$$a_n := \Lambda_\tau(\mu_n,\nu_n) \qquad{\rm and}\qquad 
b_n := 3\lambda \sqrt{ 2W_2^2(\mu_{n-1},\mu_n) + 2W_2^2 (\nu_{n-1},\nu_n)}\ .$$
Since
\begin{equation}\label{poslam2}
W_2^2(\mu,\mu_\tau)  \leq 2\tau[E(\mu) - E(\mu_\tau)]\ ,
\end{equation} 
$\sum_{k=1}^n b_k^2$ is bounded by a telescoping sum: 
$\sum_{k=1}^n b_k^2 \leq \tau 36\lambda^2[ (E(\mu) - E(\mu_n)) + (E(\nu) - E(\nu_n))]\ .$
In case $E$ is bounded below, we may assume without loss of generality that $E$ is non-negative. Then, 
\begin{align} \label{exponentialdecay}
\Lambda^{1/2}_\tau(\mu_n,\nu_n) \leq  (1+\lambda\tau)^{-n} \Lambda^{1/2} _\tau(\mu,\nu)  + \lambda \tau \frac{6(1+\lambda\tau)}{\sqrt{2\lambda}}\sqrt{E(\mu)+E(\nu)}\ .
\end{align}
Thus, for positive $\lambda$ and sufficiently small $\tau$, $\Lambda^{1/2}_\tau(\mu_n,\nu_n)$ decays ``exponentially fast'' at rate $\lambda$
up to the time that this quantity becomes $ {\mathcal O}(\tau)$.\footnote{At this point, we may use the bound $E(\mu_n)~\leq~(1~+~\lambda~\tau)^{-2n}~E(\mu)$ \cite[Theorem 3.1.6]{AGS} and apply (\ref{exponentialdecay}) iteratively.}

The proof of 
Lemma~\ref{yoco2} is elementary, so we provide it here, closing this section.

\begin{proof}[Proof of Lemma~\ref{yoco2}] Multiply both sides of (\ref{poslam3}) by $(1+\tau \lambda)^{2n-1}$ to obtain
$$(1+\tau\lambda)^{2n}a_n\leq 
(1-(\tau\lambda)^2)(1-\tau\lambda)^{2n-2}a_{n-1} + \tau
\left((1+\tau\lambda)^{2n-2}a_{n-1}\right)^{1/2}(1 + \tau\lambda)^nb_n\ .$$
Defining
$$\widetilde a_n := (1+\tau\lambda)^{2n}a_n \qquad{\rm and}\qquad 
\widetilde b_n :=  \tau(1 + \tau\lambda)^nb_n\ ,$$
we have
${\displaystyle \widetilde a_n \leq  \widetilde a_{n-1} + \widetilde a_{n-1}^{1/2} \widetilde b_n}$,
and therefore
${\displaystyle \widetilde a_n \leq  a_0 +\sum_{k=1}^n  \widetilde a_{k-1}^{1/2} \widetilde b_k}$.
Defining
$$c_n := \max\{ \widetilde a_k\ :\ 0\leq k \leq n\}\ ,$$
we have 
${\displaystyle c_n \leq a_0 +  c_n^{1/2} \sum_{k=1}^n\widetilde b_k}$.
This quadratic inequality implies that
${\displaystyle c_n^{1/2} \leq  a_0^{1/2} +  \sum_{k=1}^n\widetilde b_k}$.
By the Cauchy-Schwarz inequality, and the fact that for $\alpha := (1+\lambda \tau)^2 \geq 1$, $\sum_{k=1}^n \alpha ^k \leq \frac{\alpha}{\alpha -1}\alpha^n$, 
$$\sum_{k=1}^n \widetilde b_k \leq  \frac{\sqrt{\tau}(1+\lambda\tau)^{n+1}}{\sqrt{2\lambda}}\left(\sum_{k=1}^n b_k^2\right)^{1/2}\ .$$
\end{proof}

\section{Generalized Convexity and the Proximal Map} 

\subsection{Generalized Geodesics} \label{gengeosec}

In a Hilbert space, $x \mapsto \frac{1}{2} ||x-y||^2$ is 1-convex along geodesics. However, the same is not true for the squared 2-Wasserstein distance when the dimension of the underlying space is greater than or equal to 2 \cite[Example 9.1.5]{AGS}. Instead, Ambrosio, Gigli, and Savar\'e observe that $\mu \mapsto \frac{1}{2}W_2^2(\mu, \nu)$ is convex along a different set of curves, which we now describe.

Fix $\mu_1, \mu_2, \mu_3 \in \P_{\mu_0}(\Rd)$ with optimal plans $\bmu_{1,2} \in \Gamma_0(\mu_1,\mu_2) \ , \ \bmu_{1,3} \in \Gamma_0(\mu_1,\mu_3)$. For $1~\leq~i~<~j~\leq~3$, let $\pi_{i,j}$ be the projection onto the $i$th and $j$th components of $\Rd \times \Rd \times \Rd$. Fix $\bmu \in \P(\Rd \times \Rd \times \Rd)$ so that $\pi_{1,2} \# \bmu = \bmu_{1,2}$ and $\pi_{1,3} \# \bmu =\bmu_{1,3}$ \cite[Lemma 5.3.2]{AGS}. (We use bold font to distinguish probability measures on $\Rd \times \Rd \times \Rd$ or $\Rd \times \Rd$ from probability measures on $\Rd$.) As in \cite[Definition 9.2.2]{AGS},  a \emph{generalized geodesic joining $\mu_2$ to $\mu_3$ with base $\mu_1$} is a curve of the form
\[ \mu_\alpha^{2 \to 3}:[0,1] \to \P(\Rd), \quad \mu_\alpha^{2 \to 3} := \left( (1-\alpha) \pi_2 + \alpha \pi_3 \right) \# \bmu . \]
In the case $\mu_1 \in P^a_2(\Rd)$ and $\mu_2, \mu_3 \in \P_2(\Rd)$, this reduces to
\[ \mu_\alpha^{2 \to 3}:[0,1] \to \P(\Rd), \quad \mu_\alpha^{2 \to 3} = \left( (1-\alpha) \mathbf{t}_{\mu_1}^{\mu_2} + \alpha \mathbf{t}_{\mu_1}^{\mu_3} \right) \# \mu_1. \]
Ambrosio, Gigli, and Savar\'e demonstrate that  $\mu \mapsto \frac{1}{2}W_2^2(\mu, \mu_1)$ is 1-convex along any generalized geodesic $\mu_\alpha^{2 \to 3}$ with base $\mu_1$, for all $\mu^2, \mu^3 \in \P_{\mu_0}(\Rd)$ \cite[Lemma 9.2.1]{AGS}. Note that if the base $\mu_1$ equals either $\mu_2$ or $\mu_3$, $\mu_\alpha^{2 \to 3}$ is a (standard) geodesic joining $\mu_2$ and $\mu_3$. Thus, while $\mu \to \frac{1}{2}W_2^2(\mu, \mu_1)$ is not convex along geodesics (in the sense that it is not convex along \emph{all} geodesics), it is convex along \emph{some} geodesics.

\subsection{Functionals $E: \P_{\mu_0}(\Rd) \to \R \cup \{\infty\}$} \label{funsec}

Fix a Borel probability measure $\mu_0$. We consider functionals $E: \P_{\mu_0}(\Rd) \to \R \cup \{\infty\}$ that satisfy the following conditions:
\begin{itemize}

\item \textit{proper:} $ D(E) := \{\mu \in \P_{\mu_0}(\Rd) : E(\mu) < \infty \} \neq \emptyset$

\item \textit{coercive}\footnote{In the case $\mu_0 = \delta_0$, the Dirac mass at the origin, this is equivalent to the definition of coercivity in \cite{AGS}, where Ambrosio, Gigli, and Savar\'e require that there exist some $\tau_* > 0$ and $\mu_* \in \P_2(\Rd)$ such that 
\begin{align*} \inf_{\nu \in \P_2(\Rd)} \left\{\frac{1}{2 \tau_*} W_2^2(\mu_*,\nu) + E(\nu) \right\}> - \infty.
\end{align*}}: 
There exists $\tau^*>0$ such that for all $0<\tau <\tau^*, \mu \in \P_{\mu_0}(\Rd)$,
\[E_\tau(\mu) = \inf_{\nu \in \P_{\mu_0}(\Rd)} \left\{\frac{1}{2 \tau} W_2^2(\mu,\nu) + E(\nu) \right\}> - \infty. \]
As noted in \cite[Lemma 2.2.1]{AGS}, by a triangle inequality argument, it is enough to check that there exists $\tau_0 > 0$ such that 
\begin{align}
E_{\tau_0}(\mu_0) = \inf_{\nu \in \P_{\mu_0}(\Rd)} \left\{\frac{1}{2 \tau_0} W_2^2(\mu_0,\nu) + E(\nu) \right\}> - \infty. \label{coercivedef}
\end{align}

\item \textit{lower semicontinuous:} For all $\mu_n, \mu \in P_{\mu_0}(\Rd)$ such that $\mu_n \to \mu$ in $W_2$,
\[\liminf_{n \to \infty} E(\mu_n) \geq E(\mu). \]

\item \textit{$\lambda$-convex along generalized geodesics}: For any $\mu_1, \mu_2, \mu_3 \in \P_{\mu_0}(\Rd)$, there exists a generalized geodesic $\mu_\alpha^{2 \to 3}$ from $\mu_2$ to $\mu_3$ with base $\mu_1$ such that for all $\alpha \in [0,1]$,
\begin{eqnarray}
E(\mu_\alpha^{2 \to 3}) \leq (1-\alpha)E(\mu_2) + \alpha E(\mu_3) -  \alpha (1-\alpha) \frac{\lambda}{2} \int | x_2 - x_3|^2 d \bmu(x). \label{cagg}
\end{eqnarray}
Note that, for $\lambda >0$, this condition is stronger than requiring that $E(\mu_\alpha^{2 \to 3})$, considered as a real-valued function of $\alpha \in [0,1]$, be $\lambda W_2^2(\mu_2,\mu_3)$ convex, since
 \[\int | x_2 - x_3|^2 d \bmu \geq W_2^2(\mu_2, \mu_3).\]
\end{itemize}

If $E$ is $\lambda$-convex along generalized geodesics, then in particular it is \emph{$\lambda$-convex:} for any $\mu_1, \mu_2 \in \P_{\mu_0}(\Rd)$, there exists a geodesic $\mu_\alpha^{1 \to 2}$ from $\mu_1$ to $\mu_2$ such that for all $\alpha \in [0,1]$,
\[E(\mu_\alpha^{1 \to 2}) \leq (1-\alpha)E(\mu_1) + \alpha E(\mu_2) -  \alpha (1- \alpha) \frac{\lambda}{2} W_2^2(\mu_1,\mu_2).\]
This is equivalent to $E(\mu_\alpha^{1 \to 2})$, considered as a real-valued function of $\alpha \in [0,1]$, being $\lambda W_2^2(\mu_1,\mu_2)$ convex \cite[Remark 9.1.2]{AGS}.

The requirement that a functional $E: \P_{\mu_0} \to \R \cup \{ \infty \}$ be proper, coercive, lower semicontinuous, and convex along generalized geodesics is the natural analogue of the Hilbertian requirement that a functional $E: \H \to \R \cup \{ \infty \}$ be proper, lower semicontinuous, and convex.
The two differences are the addition of the coercivity assumption and the strengthening of the convexity assumption. In a Hilbert space $\H$, all functionals that are proper, lower semicontinuous, and convex are also coercive (in this sense), so the addition of the coercivity assumption is a natural way to ensure that the 2-Wasserstein Moreau-Yosida regularization is not identically $- \infty$. The convexity assumption is strengthened because convexity along generalized geodesics is the useful 2-Wasserstein analogue of Hilbertian convexity. While in a Hilbert space, $x \mapsto \frac{1}{2} ||x-y||^2$ is 1-convex along all geodesics, the same does not hold for the 2-Wasserstein metric. Requiring convexity of the functional on a larger class of curves compensates for the weaker convexity $W_2^2$.

\subsection{Further Results About the Proximal Map}

In the following theorem, we collect some key results from \cite[Theorem 4.1.2, Corollary 4.1.3]{AGS} regarding the proximal map.

\begin{thm} \label{proxthm}
Given $E: \P_{\mu_0}(\Rd) \to \R \cup \{\infty\}$ proper, coercive, lower semicontinuous, and $\lambda$-convex along generalized geodesics, fix $\tau> 0$ small enough so that $\tau \lambda >-1$. Then, for $\mu \in \overline{D(E)}$, the proximal map 
\[\mu \mapsto \mu_\tau\]
is well-defined. Furthermore, the following variational inequality holds:
\begin{eqnarray}
\frac{1}{2 \tau} \left(W_2^2(\mu_\tau,\nu) - W_2^2(\mu, \nu) \right) + \frac{\lambda}{2} W_2^2(\mu_\tau, \nu) \leq E(\nu) - E(\mu_\tau) - \frac{1}{2 \tau} W_2^2(\mu,\mu_\tau), \quad \forall \nu \in D(E). \label{varineq}
\end{eqnarray}
\end{thm}

When the proximal map is well-defined, it satisfies an Euler-Lagrange equation---a fact originally observed by Otto in \cite{O2, O}. We state this result in the framework of \cite[Lemma 10.1.2]{AGS}.

\begin{lm} \label{ELlem}
Given $E: \P_{\mu_0}(\Rd) \to \R \cup \{\infty\}$ proper, coercive, lower semicontinuous, and $\lambda$-convex along generalized geodesics, fix $\tau >0$ small enough so that $\tau \lambda > -1$. Assume that $\mu \in \overline{D(E)}$
so $\mu \mapsto \mu_\tau$ is well-defined by Theorem \ref{proxthm}. Then,
\begin{align}
\tau | \gradw E(\mu_\tau)| &\leq W_2(\mu,\mu_\tau). \label{EL}
\end{align}
We may interpret $| \gradw E(\mu_\tau)|$ as the metric slope (\ref{metricslopedef}) when $E$ and $\mu$ lack sufficient smoothness for the norm of the 2-Wasserstein gradient (\ref{normw}) to be well-defined.

On the other hand, if $\mu \in \P_2^a(\Rd)$ and both $E$ and $\mu_\tau$ are smooth enough so that the 2-Wasserstein gradient $\gradw  E(\mu_\tau)$ is well-defined by (\ref{gradw}), then
\begin{align}
\mathbf{t}_{\mu_\tau}^\mu &= \id + \tau  \grad \frac{\delta E}{\delta \rho}(\mu_\tau) \label{EL2}
\end{align}
$\mu_\tau$-almost everywhere and
\begin{align}
\tau | \gradw E(\mu_\tau)| &= W_2(\mu,\mu_\tau) . \label{EL3}
\end{align}

\end{lm}

\begin{proof}
(\ref{EL}) follows from \cite[Theorem 3.1.6]{AGS}.

(\ref{EL2}) follows from \cite[Lemma 10.1.2]{AGS} and the fact that, when $E$ is differentiable, $\grad \frac{\delta E}{\delta \rho}(\mu_\tau)$ is the unique element of its subdifferential at $\mu_\tau$.

(\ref{EL3}) follows from (\ref{EL2}) by considering the $L^2(\mu_\tau)$ norm of $\mathbf{t}_{\mu_\tau}^\mu - \id = \tau  \grad \frac{\delta E}{\delta \rho}(\mu_\tau)$.

\end{proof}

\section{Proofs of  Theorems \ref{yoco}, \ref{talthm}, and \ref{cothm} and Corollary \ref{lambda>0case}} \label{proofsoftheorems} 

We now prove the theorems and corollaries announced in the introduction, turning first to the generalized convexity of $E_\tau$. 
In a Hilbert space, if $E$ is proper, lower semicontinuous, and convex, then its Moreau-Yosida regularization $E_\tau$ is also convex. 
 It is well known that the exact analogue in the 2-Wasserstein metric is false. For lack of a reference, we provide the following example.

Fix $\mu_0 \in \P_2(\Rd)$ and define $E: \P_2(\Rd) \to \R \cup \{ \infty \}$ by 
\begin{equation}\label{hpdefA}
E(\mu):= \left\{ \begin{array}{ll}
        0 & \mbox{if $\mu = \mu_0$ }\\
        \infty & \mbox{otherwise}. \end{array} \right.
        \end{equation}
$E$ is proper, coercive, lower semicontinuous, and convex along all curves in $\P_2(\Rd)$. In particular, $E$ is convex along generalized geodesics.
By definition, 
\[ E_\tau(\mu) = \inf_{\nu \in \P_2(\Rd)} \left\{\frac{1}{2 \tau} W_2^2(\mu,\nu) + E(\nu) \right\} = \frac{1}{2\tau}W_2^2(\mu,\mu_0)
  \ .\]
By \cite[Example 9.1.5]{AGS}, when the dimension of the underlying space satisfies $d \geq 2$, $E_\tau$ is not $\lambda$-convex along geodesics for any $\lambda \in \R$.

As demonstrated by the previous example, the convexity of $E_\tau$ is related to the convexity of the squared 2-Wasserstein distance. This also holds in the Hilbertian case, where the convexity of $E_\tau$ is a consequence of the 1-convexity of $x \mapsto \frac{1}{2} ||x-y||^2$ \cite{M2}. Therefore, it is natural that our proof of the convexity inequality for $E_\tau$ requires the following convexity inequality for $W_2^2$.

\begin{lm}[Convexity inequality for $W_2^2$] \label{jointconv}
Fix three measures $\mu_1, \mu_2, \mu_3 \in \P(\Rd)$ that are a finite 2-Wasserstein distance apart.
Let $\mu_\alpha^{1 \to 3}$ be a generalized geodesic from $\mu_1$ to $\mu_3$ with base point $\mu_2$,
\[ \mu_\alpha^{1 \to 3} := ((1-\alpha)\pi_1 + \alpha \pi_3) \# \bmu \ , \]
where $\bmu \in \P(\Rd \times \Rd \times \Rd)$ satisfies $\bmu_{1,2} := \pi_{1,2} \# \bmu \in \Gamma_0(\mu_1, \mu_2) $ and $ \bmu_{2,3}:=\pi_{2,3} \# \bmu  \in \Gamma_0(\mu_2, \mu_3)$.
 Let $\mu_\alpha^{1 \to 2}$ be the geodesic from $\mu_1$ to $\mu_2$ defined by
\[\mu_\alpha^{1 \to 2} := \left( (1-\alpha) \pi_1 + \alpha \pi_2 \right) \# \bmu_{1,2} \ . \]
Then,
\begin{align}
W_2^2(\mu_\alpha^{1 \to 2}, \mu_\alpha^{1 \to 3}) \leq (1-\alpha) W_2^2(\mu_1,\mu_1) + \alpha W_2^2(\mu_2,\mu_3)- \alpha (1-\alpha)W_2^2(\mu_2, \mu_3). \label{jointconv1}
\end{align}
\end{lm}

\begin{proof}
Note that
\[\mu_\alpha^{1 \to 2} = \left( (1-\alpha) \pi_1 + \alpha \pi_2 \right) \# \bmu_{1,2} = \left( (1-\alpha) \pi_1 + \alpha \pi_2 \right) \# \bmu. \]
Then by \cite[Equation 7.1.6]{AGS},
\begin{align*}
W_2^2(\mu_\alpha^{1 \to 2}, \mu_\alpha^{1 \to 3}) &\leq \int_{\Rd \times \Rd \times \Rd} \left| \left[ (1-\alpha)\pi_1 + \alpha \pi_3 \right] - \left[ (1-\alpha) \pi_1 + \alpha \pi_2 \right] \right|^2 d \bmu  \\
&= \alpha^2 \int_{\Rd \times \Rd \times \Rd} \left|\pi_2 - \pi_3 \right|^2 d \bmu \\
&= \alpha^2 \int_{\Rd \times \Rd} \left|\pi_2 - \pi_3 \right|^2 d \bmu_{2,3} \\
&= \alpha^2 W_2^2(\mu_2,\mu_3) \\
&= (1-\alpha) W_2^2(\mu_1,\mu_1) + \alpha W_2^2(\mu_2,\mu_3)- \alpha (1-\alpha)W_2^2(\mu_2, \mu_3).
\end{align*}
\end{proof}

We now use this convexity inequality for $W_2^2$ to prove Theorem~\ref{yoco}.

\begin{proof}[Proof of Theorem~\ref{yoco}]
Since $E$ is proper, coercive, lower semicontinuous, and $\lambda$-convex along generalized geodesics for $\lambda \geq 0$, by Theorem \ref{proxthm}, the proximal map $\mu \mapsto \mu_\tau$ is well-defined for $\mu \in \overline{D(E)}$ and $\tau > 0$. Let $\mu_\alpha^{\bar{\mu} \to \mu_\tau}$ be the generalized geodesic from $\bar{\mu}$ to $\mu_\tau$ with base point $\mu$ on which $E$ satisfies equation (\ref{cagg}). Defining $\mu_1 := \bar{\mu}$, $\mu_2:= \mu$, and $\mu_3 :=\mu_\tau$, let $\mu_\alpha^{\bar{\mu} \to \mu}$ be the geodesic from $\bar{\mu}$ to $\mu$ described in Lemma \ref{jointconv}.
By Lemma~\ref{jointconv},
\begin{align*}
W_2^2(\mu_\alpha^{\bar{\mu} \to \mu}, \mu_\alpha^{\bar{\mu} \to \mu_\tau}) &\leq (1-\alpha)W_2^2(\bar{\mu},\bar{\mu}) + \alpha W_2^2(\mu,\mu_\tau) - \alpha (1-\alpha)W_2^2(\mu, \mu_\tau).
\end{align*}
This allows us to bound $E_\tau(\mu_\alpha^{\bar{\mu} \to \mu})$ from above:
\begin{align*}
E_\tau(\mu_\alpha^{\bar{\mu} \to \mu}) &= \inf_{\nu \in \P_{\mu_0}(\Rd)} \left\{\frac{1}{2 \tau} W_2^2(\mu_\alpha^{\bar{\mu} \to \mu},\nu) + E(\nu) \right\}
\\
&\leq \frac{1}{2 \tau} W_2^2(\mu_\alpha^{\bar{\mu} \to \mu},\mu_\alpha^{\bar{\mu} \to \mu_\tau}) + E(\mu_\alpha^{\bar{\mu} \to \mu_\tau}) \\
&\leq  \frac{1}{2 \tau} \left((1-\alpha)W_2^2(\bar{\mu},\bar{\mu}) + \alpha W_2^2(\mu,\mu_\tau)- \alpha (1-\alpha)W_2^2(\mu, \mu_\tau) \right) \\
&\quad+ (1- \alpha) E(\bar{\mu}) + \alpha E(\mu_\tau) -\alpha (1- \alpha)  \frac{\lambda}{2}  W_2^2(\bar{\mu},\mu_\tau) \\
&\leq (1-\alpha) E_\tau(\bar{\mu}) + \alpha E_\tau(\mu) - \alpha (1-\alpha) \left( \frac{1}{2 \tau}W_2^2(\mu, \mu_\tau) + \frac{\lambda}{2}W_2^2(\bar{\mu}, \mu_\tau) \right) \ .
\end{align*}
In the last step, we used that $(\bar{\mu})_\tau = \bar{\mu}$, since $E$ attains its minimum at $\bar{\mu}$.
Now, we apply 
\[ \alpha a^2 + \beta b^2 \geq \frac{\alpha \beta}{\alpha + \beta} (a+b)^2 \ , \text{ for } \alpha >0, \beta \geq0 \]
with $\alpha = 1/\tau$ and $\beta = \lambda$:
\begin{align*}
E_\tau(\mu_\alpha^{\bar{\mu} \to \mu}) &\leq (1-\alpha) \left( \frac{1}{2 \tau} W_2^2(\bar{\mu},\bar{\mu}) +E(\bar{\mu}) \right) + \alpha \left( \frac{1}{2 \tau}  W_2^2(\mu,\mu_\tau) + E(\mu_\tau) \right) \\
&\quad - \alpha (1-\alpha) \frac{\lambda_\tau}{2} \left( W_2(\mu, \mu_\tau) + W_2(\bar{\mu}, \mu_\tau) \right)^2 \\
&\leq (1-\alpha) \left( \frac{1}{2 \tau} W_2^2(\bar{\mu},\bar{\mu}) +E(\bar{\mu}) \right) + \alpha \left( \frac{1}{2 \tau}  W_2^2(\mu,\mu_\tau) + E(\mu_\tau) \right) - \alpha (1-\alpha)\frac{\lambda_\tau}{2} W_2^2(\mu, \bar{\mu}).
\end{align*}

Finally, since $E$ attains its minimum at $\bar{\mu}$, $(\bar{\mu})_\tau = \bar{\mu}$. Therefore,
\begin{align*}
E_\tau(\mu_\alpha^{\bar{\mu} \to \mu}) &\leq (1-\alpha) \left( \frac{1}{2 \tau} W_2^2(\bar{\mu},(\bar{\mu})_\tau) +E((\bar{\mu})_\tau) \right) + \alpha \left( \frac{1}{2 \tau}  W_2^2(\mu,\mu_\tau) + E(\mu_\tau) \right) - \alpha (1-\alpha) \frac{\lambda_\tau}{2} W_2^2(\mu, \bar{\mu}) \\
&= (1-\alpha) E_\tau(\bar{\mu}) + \alpha E_\tau(\mu) - \alpha (1-\alpha) \frac{\lambda_\tau}{2} W_2^2(\mu, \bar{\mu}).
\end{align*}
\end{proof}

We now use this convexity inequality to prove Theorem~\ref{talthm}.

\begin{proof}[Proof of Theorem~\ref{talthm}]
We first prove the Talagrand inequality. Since $E$ attains its minimum at $\bar{\mu}$, so does $E_\tau$. Therefore, (\ref{lambda2convex}) implies that, for all $\mu \in \overline{D(E)}$,
\begin{align*} 
&E_\tau(\bar{\mu}) \leq E_\tau (\mu_\alpha^{\bar{\mu} \to \mu}) \leq (1-\alpha) E_\tau (\bar{\mu}) + \alpha E_\tau (\mu) - \alpha (1-\alpha) \frac{\lambda_\tau}{2}W_2^2(\bar{\mu},\mu) \ .
\end{align*}
Rearranging gives
\begin{align*}
\alpha (1-\alpha) \frac{\lambda_\tau}{2} W_2^2(\bar{\mu},\mu) \leq \alpha \left( E_\tau (\mu) - E_\tau (\bar{\mu})\right) \ .
\end{align*}
Thus, for all $\alpha \in (0,1)$,
\begin{align*}
(1-\alpha) \frac{\lambda_\tau}{2} W_2^2(\bar{\mu},\mu) \leq E_\tau (\mu) - E_\tau (\bar{\mu}) \ .
\end{align*}
Sending $\alpha \to 0$ gives the Talagrand inequality (\ref{TalagrandInequality}).

We now prove the HWI inequality. Again by (\ref{lambda2convex}), for all $\mu \in \overline{D(E)}$,
\begin{align*}
E_\tau(\mu_\alpha^{\bar{\mu} \to \mu}) &\leq (1-\alpha) E_\tau(\bar{\mu}) + \alpha E_\tau(\mu) - \alpha (1-\alpha) \frac{\lambda_\tau}{2} W_2^2(\mu, \bar{\mu}) \ .
\end{align*}
Rearranging and using $\mu_\alpha^{\bar{\mu} \to \mu} = \mu_{1-\alpha}^{\mu \to \bar{\mu}}$ and $(1-\alpha) W_2(\mu, \bar{\mu}) = W_2(\mu, \mu_{1-\alpha}^{\mu \to \bar{\mu}})$ gives, for $\alpha \in (0,1)$,
\begin{align*}
(1-\alpha) E_\tau(\mu) - (1-\alpha) E_\tau(\bar{\mu}) &\leq E_\tau(\mu) - E_\tau(\mu_{1-\alpha}^{\mu \to \bar{\mu}}) - \alpha (1-\alpha) \frac{\lambda_\tau}{2} W_2^2(\mu, \bar{\mu})  \\
E_\tau(\mu) - E_\tau(\bar{\mu}) &\leq \frac{E_\tau(\mu) - E_\tau(\mu_{1-\alpha}^{\mu \to \bar{\mu}})}{1-\alpha} - \alpha\frac{\lambda_\tau}{2} W_2^2(\mu, \bar{\mu}) \\
E_\tau(\mu) - E_\tau(\bar{\mu}) &\leq \frac{E_\tau(\mu) - E_\tau(\mu_{1-\alpha}^{\mu \to \bar{\mu}})}{W_2(\mu, \mu_{1-\alpha}^{\mu \to \bar{\mu}})}W_2(\mu, \bar{\mu}) - \alpha \frac{\lambda_\tau}{2} W_2^2(\mu, \bar{\mu}) \ .
\end{align*}
Sending $\alpha \to 1$ gives  the HWI Inequality (\ref{HWIInequality}).
\end{proof}

Finally, we turn to the proof of Theorem~\ref{cothm}.
\begin{proof}[Proof of Theorem~\ref{cothm}]
By Theorem \ref{proxthm}, replacing $\nu$ with $\nu_\tau$,
\begin{eqnarray*}
\frac{1}{2 \tau} \left(W_2^2(\mu_\tau,\nu_\tau) - W_2^2(\mu, \nu_\tau) \right) + \frac{\lambda}{2} W_2^2(\mu_\tau, \nu_\tau) \leq E(\nu_\tau) - E(\mu_\tau) - \frac{1}{2 \tau} W_2^2(\mu,\mu_\tau).
\end{eqnarray*}
Similarly,
\begin{eqnarray*}
\frac{1}{2 \tau} \left(W_2^2(\nu_\tau,\mu) - W_2^2(\nu, \mu) \right) + \frac{\lambda}{2} W_2^2(\nu_\tau, \mu) \leq E(\mu) - E(\nu_\tau) - \frac{1}{2 \tau} W_2^2(\nu,\nu_\tau).
\end{eqnarray*}
Adding these and multiplying by $2 \tau$ gives
\begin{eqnarray*}
&& W_2^2(\mu_\tau,\nu_\tau) -W_2^2(\nu, \mu) + \lambda \tau \left[ W_2^2(\mu_\tau, \nu_\tau) + W_2^2(\mu, \nu_\tau) \right]  \leq 2 \tau \left[ E(\mu) - E(\mu_\tau) \right]-  W_2^2(\mu,\mu_\tau) - W_2^2(\nu,\nu_\tau).
\end{eqnarray*}
Symmetrically, we also have
\begin{eqnarray*}
&& W_2^2(\mu_\tau,\nu_\tau) -W_2^2(\nu, \mu) + \lambda \tau \left[ W_2^2(\mu_\tau, \nu_\tau) + W_2^2(\nu, \mu_\tau) \right] \quad \leq 2 \tau \left[ E(\nu) - E(\nu_\tau) \right]-  W_2^2(\mu,\mu_\tau) - W_2^2(\nu,\nu_\tau).
\end{eqnarray*}
Averaging gives 
\begin{eqnarray*}
&&W_2^2(\mu_\tau,\nu_\tau) -W_2^2(\nu, \mu) + \frac{\lambda \tau}{2} \left[ 2 W_2^2(\mu_\tau, \nu_\tau) + W_2^2(\mu, \nu_\tau) + W_2^2(\nu, \mu_\tau) \right]\nonumber \\ 
&&\quad \leq \tau  \left[E(\nu) - E(\nu_\tau) + E(\mu) - E(\mu_\tau) \right]  -  W_2^2(\mu,\mu_\tau) - W_2^2(\nu,\nu_\tau).
\end{eqnarray*}
This  allows us to bound the change in $\Lambda_\tau(\mu,\nu)$ from above:
\begin{eqnarray*}
\Lambda_\tau(\mu_\tau,\nu_\tau) - \Lambda_\tau(\mu,\nu) &=& W_2^2(\mu_\tau, \nu_\tau) + \frac{\tau^2}{2} |\gradw E(\mu_\tau)|^2 + \frac{\tau^2}{2} |\gradw E(\nu_\tau)|^2 \\
&\quad& - W_2^2(\mu, \nu) - \frac{\tau^2}{2} |\gradw E(\mu)|^2 - \frac{\tau^2}{2} |\gradw E(\nu)|^2 \\
&\leq& \tau  \left[E(\nu) - E(\nu_\tau) + E(\mu) - E(\mu_\tau) \right] -  W_2^2(\mu,\mu_\tau) - W_2^2(\nu,\nu_\tau)  \\
&\quad&+ \frac{\tau^2}{2} |\gradw E(\mu_\tau)|^2 + \frac{\tau^2}{2} |\gradw E(\nu_\tau)|^2 - \frac{\tau^2}{2} |\gradw E(\mu)|^2 - \frac{\tau^2}{2} |\gradw E(\nu)|^2 \\
&\quad& - \frac{\lambda \tau}{2} \left[ 2 W_2^2(\mu_\tau, \nu_\tau) + W_2^2(\mu, \nu_\tau) + W_2^2(\nu, \mu_\tau) \right] .
\end{eqnarray*}
By \cite[Equation 10.1.7, Lemma 10.1.5]{AGS} and H\"older's inequality, the $\lambda$-convexity of $E$ implies
\begin{align}
E(\nu) - E(\nu_\tau) \leq |\grad_W E(\nu)| W_2(\nu,\nu_\tau) - \frac{\lambda}{2}W_2^2(\nu,\nu_\tau) \ . \label{abovetanlineeq}
\end{align}
Combining this with the Euler-Lagrange equation (\ref{EL}), 
\begin{eqnarray*}
\Lambda_\tau(\mu_\tau,\nu_\tau) - \Lambda_\tau(\mu,\nu) &\leq& \tau|\gradw E(\nu)|W_2(\nu,\nu_\tau) +  \tau|\gradw E(\mu)|W_2(\mu,\mu_\tau)  -  W_2^2(\mu,\mu_\tau) - W_2^2(\nu,\nu_\tau)   \\
&\quad&+ \frac{1}{2} W_2^2(\mu,\mu_\tau) + \frac{1}{2} W_2^2(\nu,\nu_\tau) - \frac{\tau^2}{2} |\gradw E(\mu)|^2 - \frac{\tau^2}{2} |\gradw E(\nu)|^2 \\
&\quad& - \frac{\lambda \tau}{2} \left[ 2 W_2^2(\mu_\tau, \nu_\tau) + W_2^2(\mu, \nu_\tau) + W_2^2(\nu, \mu_\tau) \right] - \frac{ \lambda \tau}{2} \left[W_2^2(\nu,\nu_\tau) + W_2^2(\mu,\mu_\tau) \right].
\end{eqnarray*}  
Completing the square gives the result:
\begin{eqnarray*}
\Lambda_\tau(\mu_\tau,\nu_\tau) - \Lambda_\tau(\mu,\nu) &\leq& -\frac{1}{2}(\tau|\gradw E(\nu)| -W_2(\nu,\nu_\tau))^2 -\frac{1}{2}(\tau|\gradw E(\mu)| -W_2(\mu,\mu_\tau))^2 \\
&\quad& - \frac{\lambda \tau}{2} \left[ 2 W_2^2(\mu_\tau, \nu_\tau) + W_2^2(\mu, \nu_\tau) + W_2^2(\nu, \mu_\tau) +W_2^2(\nu,\nu_\tau) + W_2^2(\mu,\mu_\tau)\right]
\end{eqnarray*}
\end{proof}

\begin{proof}[Proof of Corollary~\ref{lambda>0case}]  First, we use $\lambda>0$ and
the Euler-Lagrange equation (\ref{EL}) to 
rewrite (\ref{contr}):
\begin{eqnarray*}
\Lambda_\tau(\mu_\tau,\nu_\tau) - \Lambda_\tau(\mu,\nu) &\leq& 
-\frac{1}{2}(\tau|\gradw E(\nu)| -W_2(\nu,\nu_\tau))^2 -\frac{1}{2}(\tau|\gradw E(\mu)| -W_2(\mu,\mu_\tau))^2 \\
&\quad& - \frac{\lambda \tau}{2} \left[ 2 W_2^2(\mu_\tau, \nu_\tau) + W_2^2(\mu, \nu_\tau) + W_2^2(\nu, \mu_\tau) +\tau^2 |\gradw E(\nu_\tau)|^2  + \tau^2 |\gradw E(\mu_\tau)|^2\right]\\
&=&
-\frac{1}{2}(\tau|\gradw E(\nu)| -W_2(\nu,\nu_\tau))^2 -\frac{1}{2}(\tau|\gradw E(\mu)| -W_2(\mu,\mu_\tau))^2 \\
&\quad& - \frac{\lambda \tau}{2} \left[ 2\Lambda_\tau(\mu_\tau,\nu_\tau) + W_2^2(\mu, \nu_\tau) + W_2^2(\nu, \mu_\tau) \right]\ .
\end{eqnarray*}
Rearranging terms, we have
\begin{eqnarray}\label{yoyoyo}
(1+\lambda\tau) \Lambda_\tau(\mu_\tau,\nu_\tau) &\leq& 
\Lambda_\tau(\mu,\nu)  - \frac{1}{2}(\tau|\gradw E(\nu)| -W_2(\nu,\nu_\tau))^2 -\frac{1}{2}(\tau|\gradw E(\mu)| -W_2(\mu,\mu_\tau))^2\nonumber \\
&\quad& - \frac{\lambda \tau}{2} \left[ W_2^2(\mu, \nu_\tau) + W_2^2(\nu, \mu_\tau) \right]\ .
\end{eqnarray}
By the triangle inequality,
\begin{eqnarray}
W_2^2(\mu,\nu_\tau) &\geq&  W_2^2(\mu,\nu)  + W_2^2(\nu,\nu_\tau)  -2W_2(\mu,\nu)
W_2(\nu,\nu_\tau)\nonumber\\
&\geq&
W_2^2(\mu,\nu)    -2W_2(\mu,\nu)
W_2(\nu,\nu_\tau)\nonumber\\
&\geq&
W_2^2(\mu,\nu)    -2\Lambda_\tau^{1/2}(\mu,\nu)
W_2(\nu,\nu_\tau)\ ,\nonumber\end{eqnarray}
and we have a similar bound for $W_2^2(\mu_\tau,\nu)$. 

Finally, for $\lambda\tau \leq 1$,
\begin{eqnarray}
\frac12(\tau |\nabla_W E(\mu)| - W_2(\mu,\mu_\tau))^2 &\geq&
\lambda \tau \left(\frac{\tau^2}{2}|\nabla_W E(\mu)|^2 - \tau|\nabla_W E(\mu)|W_2(\mu_\tau,\mu)\right)
\nonumber\\
&\geq& 
\lambda \tau \left(\frac{\tau^2}{2}|\nabla_W E(\mu)|^2 - \sqrt{2}\Lambda^{1/2}(\mu,\nu)W_2(\mu_\tau,\mu)\right)\ ,
\nonumber
\end{eqnarray}
and again we have the same inequality with $\mu$ in place of $\nu$. Using these inequalities in 
(\ref{yoyoyo}) we obtain the desired bound.
\end{proof}

\section{Examples and Applications}

\subsection{ Inequalities (\ref{lambda2convex}) and (\ref{contr}) are Sharp} \label{sharpsection}
Our first example shows that the inequality (\ref{lambda2convex}) from Theorem \ref{yoco} and the inequality (\ref{contr}) from Theorem \ref{cothm} are both sharp. For $\lambda \in \R$, consider the functional $E: \P_2^a(\Rd) \to \R$ defined by 
\begin{equation}\label{pot}
E(\mu) = \int \frac{\lambda x^2}{2} \dd \mu\ .
\end{equation}
As shown in \cite[Example 9.3.1]{AGS}, $E$ is proper, coercive, lower semicontinuous, and $\lambda$-convex along generalized geodesics. 

\begin{prop}
For $E$ given by (\ref{pot}), $\lambda \geq 0$, and $\tau >0$, define $\lambda_\tau : = \frac{\lambda}{1 + \lambda \tau}$. Then $E_\tau$ is $\lambda_\tau$-convex, and no more.
\label{MYsharp}
\end{prop}

\begin{prop}
For $E$ given by (\ref{pot}), $\mu, \nu \in D(E)$, and $\tau>0$ small enough so that $\lambda \tau >-1$, there is equality in (\ref{contr}). \label{equality}
\end{prop}

We first prove the following lemma. For $E$ given by (\ref{pot}), it is well-known that the proximal map is simply a scale transformation:

\begin{lm}\label{potlem} For $E$ given by (\ref{pot}), $\mu \in D(E)$, and $\tau>0$ small enough so that $\lambda \tau >-1$,  the proximal map associated to $E$ is the scale transformation
\begin{equation}\label{pot2}
\mu \mapsto (1+ \lambda \tau)^{-1}{\rm id}\# \mu
\end{equation}
where ${\rm id}(x) = x$ is the identity transformation. Moreover, for any $\mu,\nu\in D(E)$,
\begin{equation}\label{rok2}
W_2^2(\mu_\tau,\nu_\tau) = \frac{1}{(1+\lambda\tau)^2}  W_2^2(\mu,\nu) 
\end{equation}
and
\begin{equation}\label{rok1}
W_2^2(\mu,\nu_\tau) = \frac{1}{1+\lambda\tau} \left[ W_2^2(\mu,\nu) + 2 \tau\left(E(\mu) - \frac{1}{1+\lambda\tau}E(\nu)\right)\right]\ .
\end{equation}

\end{lm} 

\begin{proof}
At any $\mu \in D(E)$,
\begin{equation}
 \grad \frac{\delta E}{\delta \rho}(\mu) = \grad  \frac{\lambda x^2}{2} = \lambda x \label{funcder} \ .
 \end{equation} 
 For $\tau > 0$ small enough so that $\lambda \tau >-1$, the Euler-Lagrange equation (\ref{EL2}) becomes 
\[\mathbf{t}_{\mu_\tau}^\mu(x) = x + \lambda \tau x = (1+ \lambda \tau) x,\]
$\mu_\tau$-almost everywhere. This shows (\ref{pot2}):
\[(1+\lambda \tau)^{-1} \id \# \mu = \mu_\tau \ . \]

Next, fix $\phi: \Rd \to \R$ convex and define $\nu:= \nabla \phi\# \mu$. By uniqueness in the Brenier-McCann theorem, $\grad \phi$ is the optimal transport map from $\mu$ to $\nu$. If $\psi$ is defined by
$$\psi(x) = (1+ \lambda\tau)^{-2}\phi((1+\lambda\tau)x)\ ,$$
$\psi$ is convex and $\nabla \psi\# \mu_\tau = \nu_\tau$. Again, by uniqueness in the Brenier-McCann Theorem, $\grad \psi$ is the optimal transport map between $\mu_\tau$ and $\nu_\tau$. Consequently,
\begin{eqnarray}
W_2^2(\mu_\tau,\nu_\tau) &=& \int_{\R^d} |\nabla \psi(x) -x|^2{\rm d}\mu_\tau\nonumber\\
&=& (1+\lambda\tau)^{-2} \int_{\R^d} |\nabla \phi((1+\lambda\tau)x) -(1+\lambda\tau)x|^2{\rm d}\mu_\tau\nonumber\\
&=& (1+\lambda\tau)^{-2}  \int_{\R^d} |\nabla \phi(x) -x|^2{\rm d}\mu \nonumber \\
&=&  (1+\lambda\tau)^{-2}W_2^2(\mu,\nu)\ .\nonumber
\end{eqnarray}
This proves (\ref{rok2}).

Finally, note that if $\phi$ is convex and $\nabla \phi\# \mu = \nu$, by the definition of $W_2^2(\mu,\nu)$ and of $E$, 
\begin{equation}\label{rok3}
2 \int_{\R^d} x\cdot \nabla \phi(x){\rm d}\mu  = \frac{2}{\lambda}(E(\mu)+E(\nu)) - W_2^2(\mu,\nu)\ .
\end{equation}
Using that
\[(1+\lambda\tau)^{-1}\nabla \phi\# \mu = \nu_\tau \ ,\]
we may argue as above to show
\begin{eqnarray}
W_2^2(\mu,\nu_\tau) &=& \int_{\R^d} |(1+\lambda\tau)^{-1}\nabla \phi(x) -x|^2{\rm d}\mu \nonumber\\
&=& \frac{2}{\lambda}(1+\lambda\tau)^{-2}E(\nu) + \frac{2}{\lambda}E(\mu) - 2(1+\lambda\tau)^{-1} \int_{\R^d} x\cdot \nabla \phi(x){\rm d}\mu\ .\nonumber
\end{eqnarray}
Combining this with (\ref{rok3}) proves (\ref{rok1}).
\end{proof}

\begin{proof}[Proof of Proposition~\ref{MYsharp}:]
We first  explicitly compute the Moreau-Yosida regularization of $E$.
It follows from (\ref{pot2}) and the definition of $E$ that for all $\mu\in D(E)$ and $0< \tau < \infty,$
\begin{align}
W_2^2(\mu,\mu_\tau) = 2 \lambda\tau^2 E(\mu_\tau) \ .\label{mumutau}
\end{align}
Again by (\ref{pot2}),
\begin{equation}\label{rok4}
E(\mu_\tau) = (1+\lambda\tau)^{-2}E(\mu)\ .
\end{equation}
Hence,
\[E_\tau(\mu) = \frac{1}{2 \tau} W_2^2(\mu, \mu_\tau) + E(\mu_\tau) = (1+ \lambda \tau) E(\mu_\tau)  = \frac{1}{1+ \lambda \tau} E(\mu) \ .\]
Thus, the Moreau-Yosida regularization of $E$ in this (already very regular) case simply multiplies $E$ by a constant. 

It is a standard result (see e.g. \cite{AGS}) that $E$ is $\lambda$-convex, and no more.  (Its Hessian with respect to the $W_2$ Riemannain metric is $\lambda$ times the identity.)
It then follows immediately from  ${\displaystyle E_\tau(\mu) = \frac{1}{1+\lambda \tau}E(\mu)}$ that
$E_\tau$ is no more than $\lambda_\tau$-convex. 
\end{proof}

\begin{proof}[Proof of Proposition~\ref{equality}:]
We proceed by using Lemma \ref{potlem} to express quantities appearing on either side of (\ref{contr}) in terms of $W_2^2(\mu,\nu)$, $E(\mu)$ and $E(\nu)$. 
By the symmetry of $\mu$ and $\nu$, equations (\ref{rok2}) and (\ref{rok1}) allow us to express $W_2^2(\mu_\tau,\nu_\tau)$, $W_2^2(\mu,\nu_\tau)$ and $W_2^2(\nu,\mu_\tau)$ in these terms. 
By (\ref{EL3}), (\ref{funcder}), (\ref{mumutau}), and (\ref{rok4}),
\[\tau^2 | \gradw E(\mu) |^2 = \tau^2  \int (\lambda x)^2 d \mu = 2\lambda \tau^2 E(\mu)\quad{\rm and}\quad  \tau^2 | \gradw E(\mu_\tau) |^2 = W_2^2(\mu,\mu_\tau) = 2\lambda \tau^2 E(\mu) / (1+\lambda\tau)^2 \ .\]
Symmetric identities hold with $\nu$ in place of $\mu$.

Finally, direct calculation shows that both sides of  (\ref{contr}) are equal to
$$-\frac{2\lambda\tau + \lambda^2\tau^2}{(1+\lambda\tau)^2} \left[ W_2^2(\mu,\nu) + \lambda\tau^2(E(\mu)+E(\nu))\right]\ .$$
\end{proof}

As we see from (\ref{rok2}), the proximal map for $E$ is always contracting in the $W_2$ metric for $\lambda>0$. Thus, in this example, the additional terms
in $\Lambda_\tau$ are not required to produce contraction. The point of this example is rather to show that (\ref{lambda2convex}) and (\ref{contr}) are sharp.   

\subsection{The Discrete Gradient Flow for Entropy and R\'enyi Entropies} \label{Barenblattsection}
In our second example, we consider functionals $E_p$ corresponding to the entropy and R\'enyi entropies. We apply Theorem~\ref{cothm} to obtain a sharp bound, {\em uniformly} in the steps of the discrete gradient flow sequence, on the rate at which rescaled solutions of the discrete gradient flow converge to certain limiting densities, known as {\em Barenblatt densities}.  This result mirrors a well-known result obtained by Otto for the corresponding continuous gradient flow. In carrying out this analysis, we learn that the discrete gradient flow is surprisingly well-behaved, not only on average, but also uniformly in the steps. We also show that Otto's beautiful sharp results for the continuous gradient flow can be obtained very efficiently from the analysis of the discrete flow.

First, we define the functionals to be considered.
For $p > 1 - 1/d$,\footnote{The borderline case $p = 1 - 1/d$ is   more involved, and, for the sake of simplicity, we do not consider it in this paper. It may be possible to extend our approach to this case using the regularization techniques developed in \cite{BCC}.} define $U_p: \R_+ \to \R$  by
\[ U_p(s) := \left\{ \begin{array}{ll}
        \frac{s^p-s}{p-1} & \mbox{if $p \neq 1$ }\\
       s \log s & \mbox{if $p = 1$}. \end{array} \right. \]
Let $\P_2^a(\Rd)$ be the set of probability measures with finite second moment that are absolutely continuous with respect to the Lebesgue measure. Define the functional $E_p: \P(\Rd) \to \R \cup \{ \infty \}$ by
\[ E_p(\mu) := \left\{ \begin{array}{ll}
         \int_\Rd U_p(f(x))dx & \mbox{if }\mu \in \P_2^a(\Rd), d \mu(x) = f(x) dx\\
       \infty & \mbox{otherwise}. \end{array} \right. \]

For $p=1$, $E_p$ is minus the entropy. For $p \neq 1$, $E_p$ is minus the R\'enyi entropy.
As shown in \cite[Example 9.3.6]{AGS}, $E_p$ is proper, lower semicontinuous, and convex along generalized geodesics. As for coercivity, for $p>1$, $E_p$ is bounded below by $-1/(p-1)$, hence coercive. For $1- 1/d < p < 1$, $E_p$ is not bounded below, since 
$\int_{\R^d}f^p(x){\rm d}x$ can be arbitrarily large. $E_1$ is neither bounded above nor below. Nevertheless, $E_p$ is coercive for  $p > 1 - 1/d,$ when $d \geq 2$, and for $p > 1/3$, when $d=1$. Later, we shall need some of the estimates
that imply this, so we now explain this case. The $p=1$ case can be found in \cite{JKO}.

By H\"older's inequality, with exponents $1/p$ and $1/(1-p)$, for all $\nu \in \P_2^a(\Rd)$ with ${\rm d} \nu = f(x){\rm d}x$,
\begin{align*}
\int_{\R^d}f^p(x){\rm d}x &=  \int_{\R^d}f^p(x)(1+|x|^2)^p(1+|x|^2)^{-p} {\rm d}x \\ 
&\leq \left(\int_{\R^d}f(x)(1+|x|^2){\rm d}x\right)^p 
\left(\int_{\R^d}(1+|x|^2)^{-p/(1-p)}{\rm d}x\right)^{1-p} \ .\end{align*}

Furthermore, $\int_{\R^d}f(x)|x|^2{\rm d}x = \int_{\R^d}|x|^2{\rm d}\nu = W_2^2(\nu,\delta_0)$, where $\delta_0$ is the Dirac mass at the origin. By the triangle inequality, for any $\mu\in \P_2^a(\Rd)$,
$$W_2(\nu,\delta_0) \leq W_2(\mu,\nu) + W_2(\mu,\delta_0)\ ,$$
so that
$$\int_{\R^d}f^p(x){\rm d}x  \leq  \left(\int_{\R^d}(1+|x|^2)^{-p/(1-p)}{\rm d}x\right)^{1-p}( 1 + ( W_2(\mu,\nu) + W_2(\mu,\delta_0))^2)^{p}\ .$$
Finally, defining
$$C_p :=  \frac{1}{1-p}\left(\int_{\R^d}(1+|x|^2)^{-p/(1-p)}{\rm d}x\right)^{1-p}\ ,$$
we have for all $\mu,\nu\in \P_2^a(\Rd)$,
\begin{align}
E_p(\nu) \geq  - C_p \left( 1 +   2\int_{\R^d}|x|^2{\rm d}\mu + 2W_2^2 (\mu,\nu) \right)^{p}\ . \label{Eplowerbound}
\end{align}
Thus,  for all $\mu,\nu\in \P_2^a(\Rd)$,
\begin{equation}\label{rok51}
\frac{1}{2\tau}W_2^2(\mu,\nu) + E_p(\nu) \geq   \frac{1}{2\tau}W_2^2(\mu,\nu)  -   C_p \left( 1 +   2\int_{\R^d}|x|^2{\rm d}\mu + 2W_2^2 (\mu,\nu) \right)^{p}\ .
\end{equation}
For fixed $\mu$, the right hand side is bounded below for all $\tau>0$ and $\nu \in P_2^a(\Rd)$, hence $E_p$ is coercive.

Note that the condition $p > 1 - 1/d,$ when $d \geq 2$, and $p > 1/3$, when $d=1$, is exactly the condition to ensure $C_p$ is finite, and it is easy to see that coercivity fails when this is not the case. For a more general result, see \cite[Remark 9.3.7]{AGS}.

From this analysis, we may also extract an upper bound on $W_2^2(\mu,\mu_\tau)$ which will be useful later.

\begin{lm}[Distance bound for the proximal map]\label{dblem} If $d \geq 2$, fix $p > 1 - 1/d$, and if $d=1$, fix $p > 1/3$. Let $\mu\in D(E_p)$ and 
$$M(\mu) := 1 + 2\int_{\R^d}|x|^2{\rm d}\mu\ .$$
Then for all $\tau$ small enough that
$4pC_p \tau < 1$,
$$W_2^2(\mu,\mu_\tau) \leq 2\tau\frac{E_p(\mu) + C_pM(\mu)}{1 - 4 pC_p \tau}\ .$$
A similar, but more complicated, bound in terms of the same quantities holds for all $\tau>0$. 
\end{lm}

\begin{proof} By the definition of the proximal map, taking $\nu = \mu$ in the variational problem (\ref{W2proxdef}), we obtain
$$E_p(\mu) \geq \frac{1}{2\tau} W_2^2(\mu,\mu_\tau) + E_p(\mu_\tau)\ .$$
Then, by (\ref{rok51}) with $\nu = \mu_\tau$ and Bernoulli's inequality, $(1+u)^p \leq 1+pu$, 
\begin{eqnarray}E_p(\mu) &\geq &  \frac{1}{2\tau}W_2^2(\mu,\mu_\tau)  -   C_p \left( M(\mu) + 2W_2^2 (\mu,\mu_\tau) \right)^{p}\nonumber\\
&=&  \frac{1}{2\tau}W_2^2(\mu,\mu_\tau)  -   C_pM^{p}(\mu)  \left( 1 + \frac{2W_2^2 (\mu,\mu_\tau)}{M(\mu)} \right)^{p}\nonumber\\
&\geq&  \frac{1}{2\tau}W_2^2(\mu,\mu_\tau)  -   C_pM^{p}(\mu)  \left( 1 + p\frac{2W_2^2 (\mu,\mu_\tau)}{M(\mu)} \right)\ .\nonumber\\
&\geq& \left[ \frac{1}{2\tau} -  2pC_p \right]W_2^2(\mu,\mu_\tau)  -  C_pM(\mu) \ .\nonumber
\end{eqnarray}
In the last line, we used that $M(\mu) \geq 1$. 

The bound is simple due to the use of Bernoulli's inequality $(1+u)^p \leq 1+pu$. Avoiding this, one obtains a bound without restriction on $\tau$. Since we are mostly concerned with
small $\tau$, we leave the details to the reader. 

\end{proof}

If $d \geq 2$, fix $p > 1 - 1/d$, and if $d=1$, fix $p > 1/3$. Then, $E_p$ is proper, coercive, lower semicontinuous, and convex along generalized geodesics. Therefore, Theorem \ref{proxthm} guarantees that the proximal map and discrete gradient flow (\ref{discretegradflow}) are well-defined for $0< \tau < \infty$, $\mu_0 \in \overline{D(E_p)}$. Before turning to the long-time asymptotics of the discrete gradient flow for $E_p$,  we first investigate the contraction properties of  $\Lambda_\tau(\mu,\nu)$ under the proximal map. 

Unlike the functional considered in Section \ref{sharpsection}, $E_p$ is translation invariant. Specifically, for fixed $x_0 \in \Rd$, if $T_{x_0}$ is the translation given by
\[ T_{x_0} \mu := (\id - x_0) \# \mu \ , \]
then $E_p(T_{x_0} \mu) = E_p(\mu)$. The 2-Wasserstein distance is also translation invariant: for any $\mu,\nu \in \P_2^a(\Rd)$
\[W^2_2(\mu,\nu) =  W^2_2(T_{x_0}\mu,T_{x_0}\nu) \ .\]
Consequently, the proximal map associated to $E_p$ commutes with translations:
\[ (T_{x_0} \mu)_\tau = T_{x_0} (\mu_\tau)  \ .\]

On one hand, this implies that the proximal map does not contract strictly in $W_2^2$: for any $\nu \in \P_2^a(\Rd)$, $W_2^2(\nu,T_{x_0}\nu) = x_0^2$, so
\[W_2^2(\mu_\tau, (T_{x_0}\mu)_\tau) = W_2^2(\mu,T_{x_0} \mu)\ . \]
On the other hand, because the functional $E_p$ is strictly convex \cite{AGS, O}, strict inequality holds in (\ref{abovetanlineeq}) and hence in (\ref{controfa}) of Theorem \ref{cothm}:
\begin{eqnarray*}
\Lambda_\tau(\mu_\tau, \nu_\tau) < \Lambda_\tau(\mu,\nu) \ .
\end{eqnarray*}
Therefore, $\Lambda_\tau(\mu,\nu)$ is strictly decreasing under the proximal map, even though $W_2^2(\mu,\nu)$ is not.

We now turn to the long-time asymptotics of the discrete gradient flow for $E_p$.
As shown by Otto \cite{O}, the $\tau \to 0$ limit of the discrete gradient flow tends to the continuous gradient flow on $\P_2^a(\Rd)$, which corresponds to the porous medium equation or the fast diffusion equation: 
\begin{equation}\label{pme}
\frac{{\partial}}{{\partial t}}\rho(t,x) = \Delta \rho(t,x)^p\ .
\end{equation}
(For $p<1$ this is the fast diffusion equation. For $p>1$, it is the porous medium equation.)
We show that for each $\tau > 0$, the discrete flow is a strikingly close analogue of the continuous flow.  

A key feature of (\ref{pme}) is that it has {\em self-similar scaling solutions} known as {\em Barenblatt solutions}, 
\begin{equation}\label{pme2}
\sigma_p(t,x) := t^{-d\beta}h_p\left(\frac{x}{t^\beta}\right)\ ,
\end{equation}
where
\begin{equation}\label{pme3}
\beta := \frac{1}{2+d(p-1)}\  ,
\end{equation}
and
\begin{equation}\label{hpdef}
h_p(x):= \left\{ \begin{array}{ll}
        (\lambda+\frac{1-p}{p} \frac{\beta}{2}|x|^2)^{1/(p-1)} & \mbox{if $1- \frac{1}{d}<p < 1$ }\\
        \lambda \ e^{-\beta |x|^2 / 2} & \mbox{if $p = 1$}\\
        (\lambda+\frac{1-p}{p} \frac{\beta}{2}|x|^2)_+^{1/(p-1)} & \mbox{if $p > 1$}, \end{array} \right.
        \end{equation}
with normalizing constants $\lambda = \lambda(d,p)$ so that $\int_\Rd d\sigma_p(x) = \int_\Rd h_p(x) dx = 1$.

\begin{defi}[Barenblatt density] If $\mu$ is a probability measure of the form $d \mu = \sigma_p(t,x) dx$, we call $\mu$ a \emph{Barenblatt density}. Going forward, we will simply  write $\mu = \sigma_p(t,x) dx$.
\end{defi}

We now show that the Barenblatt densities are preserved under the discrete gradient flow. Before stating the next proposition, let us observe that $0 < \beta < 1$ for all values of 
$p > 1 -1/d$. Thus, the function
$s\mapsto s^\beta - \tau\beta s^{\beta-1}$
is strictly monotone increasing for $s \geq 0$ and yields the value $0$ for $s=  \tau\beta$. Consequently, for any $r>0$, there is a unique $s>  \tau\beta$ such that 
\begin{equation}\label{seq}
r^\beta = s^\beta - \tau \beta s^{\beta-1}\ .
\end{equation}

\begin{defi}[Proximal time-shift function]\label{ptsf} Define the proximal time-shift function $\theta_\tau:~\R_+~\to~\R_+$ so that, for any $r>0$,
$\theta_\tau(r)$ is the unique value of $s$ that solves (\ref{seq}).
\end{defi}

We have already observed that $\theta_\tau(r) >  \tau \beta $ for all $r>0$. Since  $r^\beta - \tau  \beta  r^{\beta-1} < r^\beta$ for all $r>0$, $\theta_\tau(r) > r$. The following lemma generalizes a result in \cite{CG} for the $p=1$ case,  showing that
the proximal map for the functional $E_p$ takes $\sigma_p(r,x){\rm d}x$ to  $\sigma_p(\theta_\tau(r),x){\rm d}x$. Thus the proximal map takes a Barenblatt density to a Barenblatt density with a larger ``time parameter''. Given that the class of Barenblatt densities is  preserved at the discrete level, we would of course expect the time parameter to increase. 

\begin{prop}\label{barenblatt} If $d \geq 2$, fix $p > 1 - 1/d$, and if $d=1$, fix $p > 1/3$. Let $\mu$ be a Barenblatt density, i.e. $\mu =  \sigma_p(r,x){\rm d}x$ for some $r>0$.
Then,  for $\tau >0$, the image of $\mu$ under the proximal map for $E_p$ is of the form 
\begin{equation}\label{renyi1}
\mu_\tau =    \sigma_p(\theta_\tau(r),x){\rm d}x\ .
\end{equation}

\end{prop}
\begin{proof} 
Given a Barenblatt density $\mu = \sigma_p(r,x)dx$ for some $r>0$, let $s := \theta_\tau(r)$ and $\nu := \sigma_p(s,x){\rm d}x$. We compute
\begin{align}
\grad \frac{\delta E_p}{\delta \rho}(\nu)  = U_p''(\sigma_p(s,x)) \grad \sigma_p(s,x) = p \sigma_p(s,x)^{p-2} \grad \sigma_p(s,x)(x) = -\frac{\beta x}{s} \quad \text{$\nu$-almost everywhere,}\label{hcalc}
\end{align}
Next, note that since  $s = \theta_\tau(r) >  \tau \beta $,
\begin{align*}
\grad \varphi(x) := x + \tau \grad \frac{\delta E_p}{\delta \rho}(\nu)  = \left(1- \frac{\tau \beta }{s}\right)x
\end{align*}
is the gradient of a convex function. Consequently, if we define
\begin{align}
 \rho := \grad \varphi \# \nu \ , \label{mu}
 \end{align}
uniqueness in the Brenier-McCann Theorem guarantees that $\grad \varphi$ is the optimal transport map between $\nu$ and $\rho$. Since
${\displaystyle \grad \varphi = \mathbf{t}_\nu^\rho = \id + \tau \grad \frac{\delta E_p}{\delta \rho}(\nu)}$ is the
the Euler-Lagrange equation (\ref{EL2}), $\nu = \rho_\tau$, the image of $\rho$ under the proximal map. With the explicit form of $\nabla \varphi$ and $\sigma_p(s,x)$, we compute
$$ \rho = \left(1-\frac{\tau \beta}{s}\right)^{-d} \sigma_p\left(s,\left(1-\frac{\tau \beta}{s}\right)^{-1}x\right){\rm d}x  =  \sigma_p\left(\left(1-\frac{\tau \beta}{s}\right)^{1/\beta}s,x\right){\rm d}x\  .$$
By definition of $s = \theta_\tau(r)$
$$
r = \left(1-\frac{\tau \beta}{s}\right)^{1/\beta}s\ .
$$
Therefore, $\rho = \sigma_p(r,x){\rm d}x = \mu $, so $\mu_\tau = \rho_\tau = \nu = \sigma_p(s,x){\rm d}x = \sigma_p(\theta_\tau(r),x){\rm d}x$. 
\end{proof}

Note that when $\tau$ is very small compared to $t>0$, and hence also compared to $s:= \theta_\tau(t)$,
$$t =   \left(1-\frac{\tau \beta}{s}\right)^{1/\beta}s \approx s - \frac{\tau \beta}{\beta} = s - \tau\ ,$$
so $\theta_\tau(t) \approx t + \tau$. Thus, in this approximation, the proximal map shifts the time forward by $\tau$, independent of $t$.  To the extent this is accurate, it makes it very easy to understand the discrete gradient flow for $E_p$ starting from a Barenblatt density: at the $n$th step of size $\tau$, one gets a Barenblatt density whose time parameter has been increased by approximately $n \tau$.   The following lemma allows us to control this approximation in precise terms. 

\begin{lm}\label{stepsize} Fix $r >0$. Then, for all $t \geq r$,
\begin{equation}\label{rok11}
\left(\frac{r }{r+ \tau}\right) \tau \   \leq\  \theta_\tau (t)  - t \ \leq\  \tau\end{equation}
\end{lm}

\begin{proof} Let $s :=  \theta_\tau(t)$ for any $t\geq r$. We recall that $0 < \beta < 1$ for all $p > 1-1/d$.
By the definition of $\theta_\tau$, we have 
$$t^{\beta} =  s^{\beta} -\tau \beta s^{\beta -1}\ .$$
Assume $s > t+\tau$. Then, by Bernoulli's inequality $(1+u)^{1-\beta} \leq  (1+(1-\beta)u)$ with $u :=  \tau/t$,
\begin{align*}
t^\beta = s^{\beta} -\tau \beta s^{\beta -1} > (t+\tau)^\beta - \tau \beta (t+\tau)^{\beta-1} = (t+ \tau)^{\beta -1}(t + (1-\beta)\tau) = t^\beta(1+u)^{\beta-1}(1+ (1-\beta) u) \geq t^\beta \ .
\end{align*}
This is a contradiction. Therefore,  $\theta_\tau(t) = s \leq  t+\tau$, which proves the upper bound in (\ref{rok11}).

To obtain the lower bound,  we use the upper bound on $s$ and the relation
${\displaystyle s = t\left(1-\tau \beta / s\right)^{-1/\beta}}$ to obtain ${\displaystyle s \geq t \left(1-\frac{ \tau \beta}{t + \tau}\right)^{-1/\beta}}$. 
Then since 
 $(1+u)^{-1/\beta} \geq 1-u/\beta$ and $t\geq r$, 
$$s \geq t\left(1 + \frac{1}{\beta}\frac{ \tau \beta}{t+ \tau}\right) \geq t+ \tau\left(\frac{r}{r + \tau}\right)\ .$$
\end{proof}

We may now use Theorem \ref{cothm} to control the rate at which rescaled solutions to the discrete gradient flow converge to a Barenblatt density. First, we define the rescaled discrete gradient flow. For any positive integer $n$, let $\theta^n_\tau$ be the $n$-fold power of $\theta_\tau$. For $t>0$, let $S_{t}$ denote the scaling transformation
given by 
$$S_t \nu = \frac{{\rm id}}{t^\beta} \# \nu \ .$$
Since  $t^{-\beta} x$ is the gradient of a convex function, uniqueness in the Brenier-McCann Theorem implies that it is the optimal transport map from $\nu$ to $S_t \nu$.

Let $\mu$ be a Barenblatt density, i.e., $\mu = \sigma_p(r,x){\rm d}x$ for some $r>0$. Then $S_r  \mu = h_p(x) dx$. Let $\{\mu_n\}$ be the discrete gradient flow with initial data $\mu$ for fixed $\tau >0$.
By Proposition~\ref{barenblatt},
$$J_\tau^n\mu = \mu_n =  \sigma_p(\theta_\tau^n(r),x){\rm d}x\ ,$$
and by definition of the scaling transformation,
\begin{equation}\label{rok5}
 S_{\theta^n_\tau(r)} J_\tau^n \mu =  S_{\theta^n_\tau(r)} \mu_n =h_p(x){\rm d}x \qquad{\rm for\ all}\ n \in \N\ .
\end{equation}
Thus, each step of the discrete gradient flow sequence is also a rescaling of $h_p(x){\rm d}x$.

In fact, something almost as good holds even when the initial data of the discrete gradient flow is not a Barenblatt density.
We apply Theorem~\ref{cothm} to prove that if $\{ \nu_n \}$ is a discrete gradient flow with initial data $\nu \in D(E_p)$ for fixed $\tau >0$, then
$$\lim_{n\to\infty} S_{\theta^n_\tau(r)} J_\tau^n \nu = \lim_{n\to\infty} S_{\theta^n_\tau(r)} \nu_n =h_p(x){\rm d}x \ .$$
That is, if you wait a while and scale the solution to view it in a fixed length scale, what you see is (essentially) a Barenblatt density, no matter what the initial data $\nu \in D(E_p)$ looked like.
Moreover, we show that $W_2(S_{\theta^n_\tau(r)} \nu_n,h_p(x){\rm d}x)$ essentially contracts at a precise polynomial rate.

\begin{thm}[Discrete fast diffusion and porous medium flow]\label{difdpm} If $d \geq 2$, fix $p > 1 - 1/d$, and if $d=1$, fix $p > 1/3$. Let $\nu \in D(E_p)$ and let $\mu = \sigma_p(r,x) dx$ for some $r>0$. Given $ 0 < \tau \leq 1$, let 
$\{\nu_n \}$ and $\{\mu_n\}$ be the discrete gradient flows (\ref{discretegradflow}) with initial conditions $\nu$ and $\mu$.  Define the rescaled  discrete gradient flow sequence
$$\widetilde \nu_n :=  S_{\theta^n_\tau(r)} \nu_n\ .$$
Then, there is an explicitly computable constant $K$ depending only on $d$, $p$, $r$, $E_p(\nu)$, and 
$$M(\nu) := 1 + 2 \int_{\R^d}|x|^2{\rm d}\nu \ ,$$
so that 
\begin{equation}\label{pseucon}
W_2^2(\widetilde \nu_n,h_p(x){\rm d}x) \leq  (\theta_\tau^n(r))^{-2\beta}[W_2(\nu,\mu)[ W_2(\nu,\mu)  +\tau^{1/2}K] +\tau K]\ .
\end{equation}
\end{thm}

From this, we readily recover Otto's
contraction result for a continuous gradient flow as follows. For any $t>0$, let ${\rm int}(t/\tau)$  denote the integer part of $t/\tau$. By 
Lemma~\ref{stepsize}, $\theta_\tau(t) = t + \tau$, up to an error that vanishes uniformly in $t$ as $\tau \to 0$. Thus, a simple iteration yields
\begin{equation}\label{convs}
\lim_{\tau\downarrow 0} \theta^{{\rm int}(t/\tau)}_\tau(r) = r+  t\ .
\end{equation}
Interpolating and taking the limit $\tau\to 0$ as in \cite{JKO}, one obtains from $\{\nu_n\}$ a solution $\rho(t,x)$ to
${\displaystyle \frac{{\partial}}{{\partial t}}\rho(t,x) = \Delta \rho(t,x)^p}$ with $\rho(0,x){\rm d}x = \nu_0$.
 Define the rescaled solution
$$\widetilde \rho(t,x) :=   ( r+ t)^{d\beta}\rho(t, (r+ t)^\beta x)\ .$$
We then conclude that, for all $t>0$, 
$$W_2^2(\widetilde \rho(t,x){\rm d}x\ ,\ h_p(x){\rm d}x) \leq (r+t)^{-2\beta}W_2^2(\rho(0,x){\rm d}x\ ,\  \sigma_p(r,x){\rm d}x )\ .$$
One may choose $r$ to minimize  $W_2^2(\rho(0,x){\rm d}x\ ,\  \sigma_p(r,x){\rm d}x )$.
Otto has shown this contraction result is sharp. Hence the ``near contraction'' result we obtain in the discrete setting cannot be improved in any manner that is uniform in $\tau$. 

Other aspects of Otto's analysis that leverage this contraction into a bound on $L^1$ convergence may be applied at the discrete level without difficulty,
and we do not go into the details here. On the other hand, while Otto proves a continuous gradient flow analogue of Theorem~\ref{cothm}, his proof does not extend to the discrete case. Theorem~\ref{cothm} provides the means to carry out the discrete analysis and to show that the discrete gradient flow analogue of (\ref{pme}) is surprisingly complete.

\begin{proof}[Proof of Theorem~\ref{difdpm}]
By Theorem \ref{cothm}, applied iteratively, we have 
\begin{equation}\label{rok61}
\Lambda_\tau(\nu_n,\mu_n) \leq \Lambda_\tau(\nu_1,\mu_1) = \Lambda_\tau(\nu_\tau,\mu_\tau)\ .
  \end{equation}
Note that we make the comparison with $\Lambda_\tau(\nu_\tau,\mu_\tau)$, not $\Lambda_\tau(\nu,\mu)$, since $ |\nabla_W E_p(\nu)|^2$ (and hence $\Lambda_\tau(\mu,\nu)$) may be infinite, but by \cite[Theorem 3.1.6]{AGS},  the strict convexity of $E$ implies
\begin{align}
| \gradw E(\nu_\tau)|^2 < | \gradw E(\nu)|^2 \label{graddecrease}
\end{align}
so $\Lambda_\tau(\nu_\tau,\mu_\tau) < \infty$.
We shall show that $\Lambda_\tau(\nu_\tau,\mu_\tau)$ is very close to $W_2^2(\nu,\mu)$, differing by a term that is ${\mathcal O}(\tau^{1/2})$.
Specifically, there exists a constant $K$ depending only $d$, $p$, $r$, $E_p(\nu)$, and $M(\nu)$, such that 
\begin{equation}\label{rok71}
\Lambda_\tau(\nu_\tau, \mu_\tau) \leq  W_2(\nu, \mu)[ W_2(\nu, \mu)  +\tau^{1/2}K] + \tau K \ .
\end{equation} 
Using this in (\ref{rok61}), we obtain
\begin{equation}\label{rok62}
  W_2^2(\nu_n, \mu_n) \leq  \Lambda_\tau(\nu_n, \mu_n) \leq W_2(\nu,\mu)[ W_2(\nu,\mu)  +\tau^{1/2}K] + \tau K\ .
  \end{equation}
  Next, by the scaling properties of the $2$-Wasserstein metric and (\ref{rok5}), for all $n\geq 1$, \ 
$$(\theta_\tau^n(r))^{-2\beta}W_2^2(\nu_n,\mu_n)  = W_2^2(S_{\theta^n_\tau(r)} \nu_n,S_{\theta^n_\tau(r)} \mu_n) =
W_2^2( \widetilde \nu_n,h_p(x){\rm d}x) \ . $$
Therefore, 
\[ W_2^2( \widetilde \nu_n,h_p(x){\rm d}x) \leq   (\theta_\tau^n(r))^{-2\beta}[W_2(\nu,\mu)[ W_2(\nu,\mu)  +\tau^{1/2}K] + \tau K] \ , \]
which is (\ref{pseucon}).

It remains to prove (\ref{rok71}). First, note that since $\mu = \sigma_p(r,x){\rm d}x$, (\ref{hcalc}) implies  ${\displaystyle \grad \frac{\delta E_p}{\delta \rho}(\mu)  = -\frac{\beta x}{r}}$. Thus, by Lemma \ref{ELlem} and the definition of the length of the gradient (\ref{normw}),
\begin{align}
\tau^2\frac{\beta^2}{r^2}\int_{\R^d}|x|^2\sigma_p(r,x){\rm d}x = \tau^2|\nabla_W E_p(\mu_\tau)|^2 =  W_2^2(\mu,\mu_\tau) \ . \label{grademutaubound}
\end{align}

We will consider the cases $p<1$, $p=1$, and $p>1$ separately. For $1- \frac1d < p<1$, when $d\geq2$, and $1/3 < p < 1$, when $d =1$, we may use the bound on $W_2(\nu,\nu_\tau)$ provided by Lemma~\ref{dblem} to show
\begin{align}
\tau^2|\nabla_W E_p(\nu_\tau)|^2 \leq  W_2^2(\nu,\nu_\tau) \leq 2\tau\frac{E_p(\nu) + C_pM(\nu)}{1 - 4 pC_p\tau}\ . \label{gradenutaubound}
\end{align}
(This particular bound requires $4pC_p \tau < 1$, but one may prove a similar bound with a more complicated constant that holds for all $\tau>0$.) 
By the triangle inequality,
\begin{eqnarray}
W_2^2(\mu_\tau,\nu_\tau) &\leq& (W_2(\mu,\nu) + W_2(\mu,\mu_\tau) + W_2(\nu,\nu_\tau))^2\nonumber\\
&\leq& W_2^2(\mu,\nu) + 2W_2(\mu,\nu)[W_2(\mu,\mu_\tau) + W_2(\nu,\nu_\tau)] + 2W_2^2(\mu,\mu_\tau) + 2W_2^2(\nu,\nu_\tau)\ .\nonumber
\end{eqnarray}
Combining this with (\ref{gradenutaubound}) and (\ref{grademutaubound}) gives
\begin{eqnarray}
\Lambda_\tau(\mu_\tau,\nu_\tau) &\leq& W_2^2(\mu,\nu) +  2W_2(\mu,\nu)\left[ \left(2\tau\frac{E_p(\nu) + C_pM(\nu)}{1 - 4\tau pC_p}\right)^{1/2} + \tau\frac{\beta}{r} \left(\int_{\R^d}|x|^2\sigma_p(r,x){\rm d}x\right)^{1/2}\right]\nonumber\\
&+& 5\tau\frac{E_p(\nu) + C_pM(\nu)}{1 - 4\tau pC_p}+ \frac{5}{2}\tau^2\frac{\beta^2}{r^2}\int_{\R^d}|x|^2\sigma_p(r,x){\rm d}x\nonumber \ .
\end{eqnarray}
This leads directly to (\ref{rok71}) with an explicit constant.

For $p>1$, by Lemma \ref{ELlem} and the definition of the proximal map,
$$\tau^2|\nabla_W E_p(\nu_\tau)|^2 \leq  W_2^2(\nu,\nu_\tau) \leq 2\tau [E_p(\nu) - E_p(\nu_\tau)] \ .$$
Since $E_p$ is bounded below, an analogous argument leads to (\ref{rok71}).

The case $p=1$ is similar to the case $p<1$; we leave the details to the reader.

\end{proof}

\medskip
\noindent{\bf Acknowledgement} We thank Luigi Ambrosio for helpful comments on a draft of this paper. We thank Haim Brezis for an enlightening conversation. We thank the anonymous reviewers for many useful suggestions.

\end{document}